\documentclass[12pt]{article}

\usepackage[margin=1in]{geometry}
\usepackage{mathpazo}
\usepackage{stmaryrd}
\usepackage[english]{babel}
\usepackage[autostyle, english = american]{csquotes}
\usepackage{anyfontsize}

\usepackage{tikz}

\usetikzlibrary{
  arrows, 
  automata, 
  shapes, 
  shapes.geometric, 
  positioning,
  calc, 
  chains, 
  decorations.pathmorphing,
  intersections, fit,cd}

\usepackage{hyperref}
\hypersetup{pdfpagemode=UseNone}

\usepackage{amssymb, amsthm, amsmath, dsfont, mathtools}
\mathtoolsset{showonlyrefs}
\usepackage{latexsym}
\usepackage{eepic}
\usepackage{sectsty}
\usepackage{framed}
\usepackage[shortlabels]{enumitem}
\usepackage{float}
\restylefloat{table}
\newtheorem{theorem}{Theorem}

\newtheorem{prop}[theorem]{Proposition}
\newtheorem{cor}[theorem]{Corollary}
\newtheorem{conjecture}[theorem]{Conjecture}
\theoremstyle{definition}
\newtheorem{definition}[theorem]{Definition}
\newtheorem{claim}[theorem]{Claim}

\newtheorem{remark}[theorem]{Remark}

\newtheorem{example}[theorem]{Example}

\newcommand{\rank}{\operatorname{rank}}
\newcommand{\krank}{\operatorname{k-rank}}

\renewcommand{\t}{{\scriptscriptstyle\mathsf{T}}}

\newcommand{\abs}[1]{\lvert #1 \rvert}
\newcommand{\bigabs}[1]{\bigl\lvert #1 \bigr\rvert}

\newcommand{\biggabs}[1]{\biggl\lvert #1 \biggr\rvert}

\renewcommand{\t}{{\scriptscriptstyle\mathsf{T}}}

\newcommand{\I}{\mathds{1}}

\newcommand{\setft}[1]{\mathrm{#1}}

\newcommand{\Lin}{\setft{L}}

\newcommand{\pro}[1]{\setft{Prod}\left(#1 \right)}

\newcommand{\complex}{\mathbb{C}}
\newcommand{\field}{\mathbb{F}}
\newcommand{\real}{\mathbb{R}}

\newenvironment{namedtheorem}[1]
	       {\begin{trivlist}\item {\bf #1.}\em}{\end{trivlist}}

\newcommand\W{\mathcal{W}}

\newcommand\V{\mathcal{V}}

\newcommand\C{\mathcal{C}}

\DeclareMathOperator{\spn}{span}

\usepackage{xifthen,mathrsfs}

\newcommand{\eql}[2]{\begin{align}\label{#1}#2\end{align}}
\newcommand{\eq}[2]{
\ifthenelse{\equal{#1}{}}{\begin{align}#2\end{align}}{\eql{#1}{#2}}}

\newcommand{\ha}[2][]{
\ifthenelse{\equal{#1}{}}{#2}{#1, #2}}

\begin{document}
\emergencystretch 3em

\title{\bf Toward a generalization of Kruskal's theorem on tensor decomposition}

\author{
  Benjamin Lovitz
  %\footnote{benjamin.lovitz@gmail.com}
  \\[2mm]
  {\it Institute for Quantum Computing and Department of Applied Mathematics}\\
  {\it University of Waterloo, Canada}}
\maketitle
\begin{abstract}
Kruskal's theorem states that a sum of product tensors constitutes a unique tensor rank decomposition if the so-called \textit{k-ranks} of the product tensors are large. In this work, we propose a conjecture in which the k-rank condition of Kruskal's theorem is weakened to the standard notion of rank, and the conclusion is relaxed to a statement on the linear dependence of the product tensors. Our conjecture would imply a generalization of Kruskal's theorem. Several adaptations and generalizations of Kruskal's theorem have already been obtained, but these results still cannot certify uniqueness when the k-ranks are below a certain threshold. Our generalization would contain several of these results, and could certify uniqueness below this threshold. We prove our conjecture over an arbitrary field $\field$ when the underlying multipartite vector space takes any one of three forms: ${\field^{d_1}\otimes \field^{d_2}}, \;{\field^{d_1}\otimes\field^{d_2}\otimes \field^2,}$ or  $\field^{d_1}\otimes \field^2 \otimes\cdots \otimes \field^2$. As a corollary to the third case, we prove that if $n$ product tensors form a circuit, then they have rank greater than one in at most $n-2$ subsystems. This is a quadratic improvement over a recent bound obtained by Ballico, and is sharp.
\end{abstract}

%-----------------------------------------------------------------------------%
\section{Introduction}\label{intro}

% so if true it could yield a novel alternate proof of this fundamental result. Note that several alternate proofs of Kruskal's theorem are already present in the literature \cite{Jiang2004KruskalsPL,STEGEMAN2007540,RHODES20101818,landsberg2012tensors}.

Let $[m]=\{1,\dots, m\}$ when $m$ is a positive integer, and let $\V_j$ be a vector space for each $j \in [m]$. A \textit{product tensor} in $\V=\V_1\otimes\dots \otimes \V_m$ is a non-zero tensor $z \in \V$ of the form ${z=z_1\otimes \dots \otimes z_m}$, with $z_j \in \V_j$ for all $j \in [m]$. We refer to the spaces $\V_j$ that make up the space $\V$ as \textit{subsystems}.
% (also known as \textit{factors} and \textit{loadings}).
The \textit{tensor rank} (or \textit{rank}) of a tensor $v \in \V$, denoted by $\rank(v)$, is the minimum number $r$ for which $v$ is the sum of $r$ product tensors. A decomposition of $v$ into the sum of $r$ product tensors is called a \textit{tensor rank decomposition} of $v$.
%(also known as a \textit{canonical decomposition (CANDECOMP), parallel factor (PARAFAC) model, canonical polyadic (CP) decomposition,} and \textit{topographic components model}).
An expression of $v$ as a sum of product tensors (not necessarily of minimum number) is known simply as a \textit{decomposition} of $v$. A decomposition of $v$ into the sum of $n$ product tensors
\begin{align}\label{decomposition_of_v}
v=\sum_{a \in [n]} x_{a,1} \otimes\dots \otimes x_{a,m}
\end{align}
is said to be the \textit{unique tensor rank decomposition of $v$} if for any other decomposition
\begin{align}\label{other_decomp}
v=\sum_{a \in [n]} y_{a,1} \otimes\dots \otimes y_{a,m}
\end{align}
of $v$ as a sum of $n$ product tensors there exists a permutation $\sigma \in S_n$ such that $x_a=y_{\sigma(a)}$ for all $a \in [n]$. It is easy to see that this implies $\rank(v)=n$. The decomposition~\eqref{decomposition_of_v} is said to be \textit{unique in the $j$-th subsystem} if for any other decomposition~\eqref{other_decomp} there exists a permutation $\sigma \in S_n$ such that $x_{a,j} \in \spn \{y_{\sigma(a),j}\}$ for all $a \in [n]$. Kruskal's theorem gives sufficient conditions for a decomposition~\eqref{decomposition_of_v} to constitute a unique tensor rank decomposition~\cite{kruskal1977three}. We refer to results of this kind as \textit{uniqueness criteria}.

Uniqueness criteria have found scientific applications in signal processing and spectroscopy, among others
%\cite{smilde2004multi,kroonenberg2008applied,doi:10.1137/07070111X,art1,comon2010handbook,art2,cichocki2015tensor,sidiropoulos2017tensor}
\cite{
%smilde2004multi,kroonenberg2008applied,doi:10.1137/07070111X,art1,comon2010handbook,
art2, landsberg2012tensors, cichocki2015tensor,sidiropoulos2017tensor}. In these circles, subsystems are also referred to as \textit{factors} and \textit{loadings}, and the tensor rank decomposition is also referred to as the \textit{canonical decomposition (CANDECOMP), parallel factor (PARAFAC) model, canonical polyadic (CP) decomposition,} and \textit{topographic components model}. Uniqueness of a tensor decomposition is also referred to as \textit{specific identifiability}, and uniqueness criteria as \textit{identifiability criteria}.

%Kruskal's theorem~\cite{kruskal1977three} was a seminal result that determined sufficient conditions for a decomposition~\eqref{decomposition_of_v} to constitute a unique tensor rank decomposition. We refer to results of this kind as \textit{uniqueness conditions}.
%\cite{kruskal1977three,1323268,10.1137/040608830,stegeman2009uniqueness,doi:10.1137/090779632,domanov2013uniqueness,domanov2013uniqueness2,domanov2014canonical,sorensen2015new,sorensen2015coupled,domanov2017canonical}.
The \textit{Kruskal-rank} (or \textit{k-rank}) of a set of vectors $\{u_1,\dots, u_n\}$, denoted by ${\krank(u_1,\dots, u_n)}$, is the largest number $k$ for which $\dim\spn\{ u_a : a \in S\} =k$ for every subset $S \subseteq [n]$ of size $\abs{S}=k$.
%\begin{definition}
%Let $n$, $m$ and $k_1, \dots, k_m$ be positive integers, and let $\V=\V_1\otimes \dots\otimes \V_m$ be a multipartite vector space over a field $\mathbb{F}$. We say that a set of product tensors
%\begin{align}\label{product_tensors}
%\{{x_{a,1}}\otimes\dots\otimes x_{a,m}: a \in [n] \}\subseteq \V
%\end{align}
%is in $(k_1,\dots,k_m)$\textit{-general position} if for each index $j\in [m]$, the vectors $\{x_{a,j}: a \in [n]\}$ have \textit{Kruskal rank} (or \textit{k-rank}) at least $k_j$.
%\end{definition}
Kruskal's theorem states that if a collection of product tensors $\{x_{a,1}\otimes \dots \otimes x_{a,m} : a \in [n]\}$ has large enough k-ranks $k_j=\krank(x_{1,j},\dots, x_{n,j})$, then their sum constitutes a unique tensor  decomposition. This theorem was originally proven for $m=3$ subsystems over $\real$ \cite{kruskal1977three}, was later extended to more than three subsystems by Sidiropoulos and Bro \cite{sidiropoulos2000uniqueness}, and then to an arbitrary field by Rhodes \cite{RHODES20101818} (Landsberg's proof also applies to an arbitrary field \cite{landsberg2012tensors}).
\begin{theorem}[Kruskal's theorem]\label{kruskal}
Let $n \geq 2$ and $m \geq 3$ be integers, let ${\V=\V_1\otimes \cdots \otimes \V_m}$ be a multipartite vector space over a field $\field$, and let
\begin{align}
\{{x_{a,1}}\otimes\dots\otimes x_{a,m}: a \in [n] \}\subseteq \V
\end{align}
be a set of product tensors. For each $j \in [m]$, let
\begin{align}
k_j = \krank(x_{1,j},\dots, x_{n,j}).
\end{align}
If ${2n \leq\sum_{j=1}^m (k_j-1)+1},$ then $\sum_{a \in [n]} x_{a,1} \otimes \dots \otimes x_{a,m}$ constitutes a unique tensor rank decomposition.
\end{theorem}

%In \cite{Jiang2004KruskalsPL,STEGEMAN2007540,RHODES20101818,landsberg2012tensors} alternate proofs of Kruskal's theorem are given.

%In \cite{Berge:2002aa} it is shown that Kruskal's inequality $2n \leq \sum_{j=1}^m (k_j-1)+1$ is also necessary for uniqueness when $n \leq 3$ and $m=3$, but not for $n\geq 4$. However, it is also shown that Kruskal's inequality is necessary for uniqueness even for $m=3$, $n=4$ under the condition that
%\begin{align}
%\dim\spn \{x_{a,j} : a \in [n]\}=k_j\quad\text{for all}\quad j \in [3].
%\end{align}
%In \cite{STEGEMAN2006210} it is shown that Kruskal's inequality is not necessary for uniqueness, even under this restrictive condition, for $n\geq 5$.

In \cite{DERKSEN2013708} it is shown that the inequality appearing in Kruskal's theorem is sharp, in the sense that there exist cases in which ${2n = \sum_{j=1}^m (k_j-1)+2}$ and the decomposition is not unique. While Kruskal's theorem gives sufficient conditions for uniqueness, necessary conditions are obtained in \cite{krijnen1993analysis,STRASSEN1983645,liu2001cramer}. In \cite{effective} it is shown that Kruskal's theorem is \textit{effective} over $\real$ or $\complex$ in the sense that it certifies uniqueness on a dense open subset of the smallest semialgebraic set containing the set of rank $n$ tensors. Generic uniqueness has been studied, for example, in \cite{Bocci:2014aa,chiantini,lathauwergeneric}. Uniqueness of symmetric tensor decompositions (also known as the INDSCAL model), and of other types of decompositions have been studied, for example, in \cite{article,Ballico:2012aa,sorensen2015coupled,Massarenti:2018aa,Angelini:2020aa}.

Our main conjecture in this work is not itself a uniqueness criterion, but would imply a criterion that generalizes Kruskal's theorem. In our main conjecture, the k-rank condition in Kruskal's theorem is relaxed to a condition on the standard rank of $(x_{1,j},\dots, x_{n,j})$. In turn, the conclusion is also relaxed to a statement describing the linear dependence of the product tensors $x_{a,1}\otimes \dots \otimes x_{a,m}$. Before stating our main conjecture, we first introduce the generalization of Kruskal's theorem it would imply.

%A drawback to Kruskal's theorem, and to all of the known uniqueness conditions that we are aware of, is that they cannot certify uniqueness when the $k$-ranks are small. The following conjectural statement would generalize Kruskal's theorem and has the capacity to efficiently certify uniqueness even when the k-ranks are small.

\begin{conjecture}\label{k-gen}
Let $n \geq 2$ and $m \geq 3$ be integers, let $\V=\V_1 \otimes \cdots \otimes \V_m$ be a multipartite vector space over a field $\field$, and let
\begin{align}
\{{x_{a,1}}\otimes\dots\otimes x_{a,m}: a \in [n] \}\subseteq \V
\end{align}
be a set of product tensors. For each subset $S \subseteq [n]$ of size $2\leq \abs{S} \leq n$ and index $j\in [m]$, let
\begin{align}
d_j^S=\dim\spn \{ x_{a,j}: a \in S\}.
\end{align}
If $\;2 \abs{S} \leq\sum_{j=1}^m (d_j^S-1)+1$ for every such $S$, then $\sum_{a \in [n]} x_{a,1} \otimes \dots \otimes x_{a,m}$ constitutes a unique tensor rank decomposition.
\end{conjecture}

To see that Conjecture~\ref{k-gen} contains Kruskal's theorem, assume the conditions of Kruskal's theorem hold and note that for any subset $S \subseteq [n]$, the product tensors $\{x_a : a \in S\}$ satisfy $d_j^S \geq \min\{k_j, \abs{S}\}$. Using this fact, it is easy to verify that ${2n \leq \sum_{j=1}^m (k_j-1)+1}$ implies ${2\abs{S} \leq \sum_{j=1}^m (d_j^S-1)+1}$ for every subset $S \subseteq [n]$ of size $2 \leq \abs{S} \leq n$.

In Section~\ref{uniqueness_applications} we compare Conjecture~\ref{k-gen} to the uniqueness criteria of Domanov, De Lathauwer, and S\o{}rensen (DLS), which generalize Kruskal's theorem in the case of three subsystems \cite{domanov2013uniqueness,domanov2013uniqueness2,domanov2014canonical,sorensen2015new,sorensen2015coupled}. Other uniqueness criteria that we are aware of can only be applied when the tensor rank $n$ is small \cite{chiantini,unknown}, or when $k_1=d_1=n$ \cite{doi:10.1137/090779632}. The question of whether a given decomposition constitutes a unique tensor rank decomposition can be phrased as an ideal membership problem, and hence is theoretically computable, but likely computationally intractable. We restrict our attention to the uniqueness criteria of DLS for which we are aware of an efficient implementation. The cited results of DLS contain many similar but incomparable criteria, which can be difficult to keep track of. In Theorem~\ref{uniqueness} we synthesize these criteria into a single statement, and directly prove a generalization of one of them. Unfortunately, Theorem~\ref{uniqueness} requires the k-ranks to be large (see \eqref{k-rank-large} for a precise statement). Our Conjecture~\ref{k-gen} does not require the k-ranks to be large, and hence has the potential to efficiently certify uniqueness for a large class of tensors that cannot be handled by current means.

Our Conjecture~\ref{k-gen} also appears to give evidence for a generalization of Theorem~\ref{uniqueness}, which would unify this medley of uniqueness criteria into a single, elegant criterion. We follow closely the formalism of \cite{domanov2013uniqueness,domanov2013uniqueness2}. Every uniqueness criterion in Theorem~\ref{uniqueness} assumes a certain condition, which we call Condition~U, that guarantees uniqueness in the first subsystem by Kruskal's permutation lemma~\cite{kruskal1977three}. (Our Condition~U is Condition~$U_{n-d_1+1}$ of \cite{domanov2013uniqueness,domanov2013uniqueness2}, with the additional assumption that $k_1\geq 2$.) Further conditions are then assumed which certify full uniqueness. As there is no known method to check Condition~U efficiently, two more restrictive, but more efficiently checkable conditions, called Condition~H and Condition~C, are often used instead. Our Conjecture~\ref{k-gen} would imply that Condition~H alone certifies uniqueness, which leads us to ask whether Condition~U alone certifies uniqueness. Such a statement would generalize Theorem~\ref{uniqueness}.

%It is known that if a set of product tensors satisfy $n \leq \sum_{j=1}^m (k_j-1)$, then they are linearly independent, and if $n+1 \leq \sum_{j=1}^m (k_j-1)$, then they contain no product tensors in their span except trivial scalar multiples. At the end of Section~\ref{uniqueness_applications} we prove that if $n+r\leq \sum_{j=1}^m (k_j-1)+1$ for any integer $r \in \{0, 1, \dots, n-1\}$, then the sum of any subset of the product tensors of size greater than $r$ has multilinear rank $(r_1,\dots, r_m)$, where $r_j > r$ for all $j \in [m]$.

At the end of Section~\ref{uniqueness_applications}, we generalize Condition~U to the case of greater than three subsystems. We also prove a related result on the multilinear rank of linear combinations of product tensors with large k-ranks, which generalizes results of Ha and Kye \cite{1751-8121-48-4-045303}.

% which we now describe using the formalism introduced by De Lathauwer and Domanov \cite{domanov2013uniqueness,domanov2013uniqueness2}. The premise of Theorem~\ref{uniqueness} 
%
% called Condition~U (see Section~\ref{uniqueness_applications} or~\cite{domanov2013uniqueness})
%
%(Condition~U, see Section~\ref{uniqueness_applications} or~\cite{domanov2013uniqueness}) is assumed which guarantees that the decomposition is unique in one subsystem by Kruskal's permutation lemma~\cite{kruskal1977three}. Then, further conditions are imposed to certify full uniqueness. Unfortunately, there is no known method to check Condition~U efficiently. As a result, two more restrictive but more efficiently checkable conditions, called Condition~H and Condition~C, are often used instead. We show that our Conjecture~\ref{k-gen} would imply that Condition~H alone certifies uniqueness. This leads us to ask whether Condition~U alone certifies uniqueness. Such a statement would generalize all of the efficiently checkable uniqueness conditions of De Lathauwer et al.

We now state our main conjecture, which in Section~\ref{uniqueness_applications} we prove would imply Conjecture~\ref{k-gen}. We first require a definition.

\begin{definition}
Let $n\geq 2$ be an integer, and let $\V$ be a vector space over a field $\field$. We say that a set of non-zero vectors $\{v_1,\dots,v_n\}\subseteq \V\setminus \{0\}$ \textit{splits as a direct sum} (or simply, \textit{splits}) if there exist non-empty sets $S, T \subseteq \{v_1,\dots,v_n\}$ that partition $\{v_1,\dots,v_n\}$ (i.e. $S \cup T = \{v_1,\dots,v_n\}$ and $S \cap T=\{\} $) such that
\begin{align}\label{sdef}
\spn\{v_1,\dots,v_n\}=\spn(S) \oplus \spn(T).
\end{align}
\end{definition}

%Here we conjecture an analogous statement to Kruskal's theorem.
%%prove it in two special cases, and show that it would imply Kruskal's theorem as a corollary. We also use Kruskal's theorem to prove a family of statements which contain recent results in \cite{1751-8121-48-4-045303}, and show that our conjecture would imply an analogous family of statements for vectors 
%To introduce our conjecture and it's connection to Kruskal's theorem, we first state (an equivalent reformulation of) Kruskal's theorem.

Now we state our main conjecture.
\begin{conjecture}\label{conjecture}
Let $n \geq 2$ and $m \geq 1$ be integers, let $\V=\V_1\otimes \cdots\otimes \V_m$ be a multipartite vector space over a field $\field$, and let
\begin{align}
R=\{x_{a,1} \otimes \dots \otimes x_{a,m}: a \in[n] \} \subseteq \V
\end{align}
be a set of product tensors. For each $j\in [m]$, let
\begin{align}
d_j=\dim\spn \{ x_{a,j}: a \in [n]\}.
\end{align}
If $n\leq \sum_{j=1}^m (d_j-1)+1$, then $R$ splits.
\end{conjecture}

We prove Conjecture~\ref{conjecture} in three special cases.
\begin{theorem}\label{s_cases}
Conjecture~\ref{conjecture} holds in the following cases:
\begin{enumerate}
\item
%\emph{\textbf{Bipartite case (see Proposition~\ref{bipartite_prop}):}} 
$m=2$ and $d_1,d_2\geq 1$ (hereafter referred to as the \textit{bipartite case}).
\item %\emph{\textbf{Restricted tripartite case (see Theorem~\ref{tripartite_thm}):}} 
$m=3$, $d_1,d_2\geq 1$, and $1\leq d_3\leq 2$ (hereafter referred to as the \textit{restricted tripartite case}).
\item %\emph{\textbf{Restricted multipartite case (see Theorem~\ref{sep_combo}):}} 
$m\geq 4$, $d_1 \geq 1$, and ${1\leq d_2,\dots,d_m \leq 2}$ (hereafter referred to as the \textit{restricted multipartite case}).
\end{enumerate}
\end{theorem}
In Section~\ref{k-gen} we prove these three statements in Proposition~\ref{bipartite_prop}, Theorem~\ref{tripartite_thm}, and Theorem~\ref{sep_combo} respectively. In Proposition~\ref{bipartite_prop} and Theorem~\ref{tripartite_thm}, we actually prove more general statements than the bipartite and restricted tripartite cases of Conjecture~\ref{conjecture}. Theorem~\ref{sep_combo} implies that Conjecture~\ref{k-gen} holds for $m\geq 4$ when for every subset $S\subseteq[n]$ of size $2 \leq \abs{S} \leq n$ there exists an index $j\in [m]$ such that $d_i^S \leq 2$ for all $i \in [m]\setminus j$. Unfortunately, we have been unable to find an example for which this statement certifies uniqueness, but some reshaping of Kruskal's theorem does not (see \cite{effective}). Theorem~\ref{tripartite_thm} similarly gives rise to a uniqueness criterion, which turns out to be contained in Kruskal's theorem.

Recall that a set of non-zero vectors forms a \textit{circuit} if it is linearly dependent and any non-empty strict subset is linearly independent. Since a circuit does not split, an immediate consequence of Theorem~\ref{sep_combo} is that if a set of $n$ product tensors forms a circuit, then $d_j>1$ for at most $n-2$ indices $j \in [m]$ (Corollary~\ref{linincor}). This quadratically improves the bound $\binom{n}{2}+n$ obtained by Ballico \cite[Theorem 1.1]{ballico2020linearly}. In Section~\ref{conjecture_kruskal} we use Derksen's result \cite{DERKSEN2013708} to  prove that our bound is sharp in the sense that there exist circuits for which $d_j>1$ for $n-2$ indices $j \in [m]$. We furthermore prove that the inequality ${n \leq \sum_{j=1}^m (d_j-1)+1}$ appearing in Conjecture~\ref{conjecture} would be sharp in a similar sense. Ballico used his result to study linearly dependent sets of product tensors (see also \cite{Ballico_2020}). In Section~\ref{consequences} we the use the (well-known) $n=3$ case of our bound to provide an alternate proof of a recent result in quantum information theory~\cite{PhysRevA.95.032308}. In a follow up work, the author studies \textit{decomposable correlation matrices}, correlation matrices that can be written as the Schur product of correlation matrices of reduced rank, and uses Corollary~\ref{linincor} to bound the number of non-trivial correlation matrices that can appear in a decomposition \cite{doi:10.1080/03081087.2019.1661347}.

In Section~\ref{consequences} we also introduce two statements that would follow from Conjecture~\ref{conjecture}, special cases of which follow from Theorem~\ref{s_cases}. First, if the sum of a set of $n$ product tensors has tensor rank at most $r$ for some $r \in [n-1]$, and $n+r\leq \sum_{j=1}^m (d_j-1)+1,$ then the sum of some subset of the product tensors of size at least $r$ has tensor rank less than $r$. The restricted multipartite case of this statement with $r=1$ is essentially Corollary~\ref{linincor}. This statement has connections with Condition~U in the study of uniqueness criteria. Second, if a tensor has multilinear rank $(r_1,\dots, r_m)$ (see Section~\ref{mp}) and tensor rank $n$ with $2n \leq \sum_{j=1}^m (r_j-1)+1$, then for any two tensor rank decompositions of it, the sum of a strict subset of the two sets of product tensors involved must agree (Corollary~\ref{ana_cor}). This conclusion can be viewed as a weakening of the statement that the tensor rank decomposition is unique.

We close this introduction by describing an equivalent formulation of Conjecture~\ref{conjecture} that we will use to prove Theorem~\ref{s_cases}, and which may be of independent theoretical interest. Consider the following definition.
\begin{definition}
Let $n\geq 2$ be an integer, and let $\V$ be a vector space over a field $\field$. We say a set of non-zero vectors $\{v_1,\dots,v_n\}\subseteq \V\setminus\{0\}$ is \textit{minimal} over $\field$ if there exist (non-zero) scalars $\alpha_1,\dots,\alpha_n \in \field$ such that
\begin{align}\label{eq:zero}
\sum_{a\in [n]} \alpha_a v_a =0,
\end{align}
and for every subset $S \subseteq [n]$ of size $1\leq \abs{S} \leq n-1$, it holds that
\begin{align}\label{eq:nonzero}
\sum_{a\in S} \alpha_a v_a \neq 0.
\end{align}
We say that \eqref{eq:zero} constitutes a \textit{minimal linear dependence} of $\{v_1,\dots,v_n\}$. 
\end{definition}
Note that our definition of minimal differs from that of Ballico \cite{Ballico_2020}. In Proposition~\ref{field} we prove that splitting is invariant under field extensions, and in Proposition~\ref{prop:min_split} we prove that over an infinite field, a set of vectors splits if and only if it is not minimal. As a result, it suffices to prove (special cases of) Conjecture~\ref{conjecture} over an infinite field, with ``splits'' replaced by ``is not minimal'' if desired. When ambiguity arises, we refer to these two versions of Conjecture~\ref{conjecture} as the splitting and non-minimal versions, respectively. In Section~\ref{special_cases}, we appeal to both of these versions to prove Theorem~\ref{s_cases}.

%-----------------------------------------------------------------------
\section{Acknowledgments}
%-----------------------------------------------------------------------

I thank Edoardo Ballico, Matthias Christandl, Harm Derksen, Dragomir {\DJ}okovi{\'{c}}, Ignat Domanov, Nathaniel Johnston, Joseph M. Landsberg, Lieven De Lathauwer, Chi-Kwong Li, Rotem Liss, Daniel Puzzuoli, William Slofstra, Hans De Sterck, and John Watrous for helpful discussions and comments on drafts of this manuscript. A previous iteration of this work \cite{tensor} contained only the non-minimal version of Conjecture~\ref{conjecture}. I thank Harm Derksen for suggesting the splitting version that appears here. I thank Dragomir {\DJ}okovi{\'{c}} for first suggesting a connection to Kruskal's theorem, for simplifying an argument in the proof of Theorem~\ref{sep_combo}, and for suggesting that these results might hold for an arbitrary field.

%-----------------------------------------------------------------------
\section{Mathematical preliminaries}\label{mp}
%-----------------------------------------------------------------------

Here we review some mathematical background for this work that was not covered in the introduction. For vector spaces $\V_1, \dots, \V_m$ over a field $\field$, we use $\pro{\V_1 : \dots : \V_m}$ to denote the set of (non-zero) product tensors in $\V_1\otimes  \dots \otimes \V_m$. This set forms an algebraic variety given by the affine cone over the {\it Segre variety} $\setft{Seg}(\mathbb{P} \V_1 \times \dots \times \mathbb{P} \V_m)$, with $0$ removed. We use symbols like $a, b$ to index tensors, and symbols like $i, j$ to index subsystems. For vector spaces $\V$ and $\W$, let $\Lin(\V,\W)$ denote the space of linear maps from $\V$ to $\W$. We use the shorthand $\Lin(\V)=\Lin(\V,\V)$. For a vector space $\V$ of dimension $d$, let $\{e_1,\dots, e_d\}$ be a standard basis for $\V$.

For a product tensor $z \in \pro{\V_1 : \dots : \V_m}$, the vectors ${z_{j} \in \V_j}$ for which ${z=z_1 \otimes \dots \otimes z_m}$ are uniquely defined up to scalar multiples $\alpha_{1} z_1, \dots, \alpha_m z_m$ such that $\alpha_1\cdots \alpha_m=1$. For positive integers $n$ and $m$, we frequently define sets of product tensors
\begin{align}
\{ x_a : a \in [n]\} \subset \pro{\V_1 : \dots : \V_m}
\end{align}
without explicitly defining corresponding vectors $\{x_{a,j}\}$ such that
\begin{align}
x_a=x_{a,1}\otimes \dots \otimes x_{a,m}\quad \text{for each} \quad a\in [n].
\end{align}
In this case, we implicitly fix some such vectors, and refer to them without further introduction.
% We also implicitly define the linear maps
%\begin{align}
%X_j=(x_{1,j}, \dots, x_{n,j}) \in \Lin(\field^n,\V_j)
%\end{align}
%for each $j \in [m]$.

 We use the notation
\begin{align}
x_{a, \hat{j}}&= x_{a,1} \otimes \dots \otimes x_{a,j-1}\otimes x_{a, j+1} \otimes \dots \otimes x_{a,m},\\
\V_{\hat{j}}&=\V_1 \otimes \dots \otimes \V_{j-1}\otimes \V_{j+1}\otimes\dots \otimes V_{m},
\end{align}
so $x_{a,\hat{j}}\in \V_{\hat{j}}$. Note that $\V_1\otimes \dots \otimes \V_m$ is naturally isomorphic to $\Lin(\V_j^*,\V_{\hat{j}})$ for any $j \in [m]$, where $\V_j^*$ is any dual vector space to $\V_j$. The rank of a tensor in $\V_1 \otimes \V_2$ is equal to the rank of the corresponding linear operator in $\Lin(\V_1^*, \V_2)$. We denote the standard matrix rank of a tensor $v\in \V$, viewed as an element of $\Lin(\V_j^*,\V_{\hat{j}})$, by $\rank_j(v)$. The \textit{multilinear rank} of $v$ is the $m$-tuple $(\rank_1(v),\dots,\rank_m(v))$.

We write $S \cup T$ to denote the union of two sets $S$ and $T$. If $S$ and $T$ happen to be disjoint, we often write $S \sqcup T$ instead to remind the reader of this fact. For a positive integer $s$, we say that a collection of subsets $S_1,\dots, S_s \subseteq T$ \textit{partitions} $T$ if $S_q \neq \{\}$ for all $q \in [s]$, $S_q \cap S_r =\{\}$ for all $q\neq r \in [s]$, and $S_1 \sqcup \dots \sqcup S_s=T$.

%We use the notation $\ip{\cdot}{\cdot}: \W\times \W \rightarrow \field$ to denote an arbitrary non-degenerate bilinear form (the underlying vector space $\W$ of $\ip{\cdot}{\cdot}$ will usually be clear from the context, and will not be specified explicitly). Let $\W, \X$ be vector spaces. We define a \textit{contraction map} associated to $\ip{\cdot}{\cdot}$ by 
%\begin{align}
%(\bra{\cdot} \otimes \I) &: \W \times ( \W \otimes \X ) \rightarrow \X\\
%(\bra{\cdot} \otimes \I) (w, u \otimes x)&:= (\bra{w} \otimes \I)(u \otimes x)= \ip{w}{u} x, \label{contract}
%\end{align}
%where the second expression in~\eqref{contract} is a notational convention, and the third is the definition. Equation~\ref{contract} completely determines $\bra{\cdot} \otimes \I$ because, as with any linear map, the action of $\bra{w} \otimes \I$ on $\W \otimes \X$ is uniquely determined by its action on product tensors.

We close this section by proving a pair of propositions that will allow us to work over an infinite field without loss of generality. The first proposition states that splitting is invariant under field extensions, and the second states that over an infinite field, splitting and non-minimality are equivalent. As a result, to prove (special cases of) Conjecture~\ref{conjecture} over an arbitrary field, it suffices to prove either the splitting or non-minimal version over an infinite field.

For a field $\field$, let $\overline{\field}$ denote any infinite field extension of $\field$ (e.g. the algebraic closure), and for a vector space $\V$ over $\field$, let $\overline{\V}=\overline{\field} \otimes_{\field} \V$ denote the corresponding extension of scalars of $\V$. If $\{v_1,\dots, v_n\}$ is linearly independent in $\V$, then $\{1 \otimes v_1,\dots, 1 \otimes v_n\}$ is linearly independent in $\overline{\V}$. 
%For $\V \cong \field^d$, we have $\overline{\V} \cong \overline{\field}^d$.
%We sometimes abuse notation and write elements $\alpha \otimes v \in \overline{\V}$ as $\alpha v$. For a dual vector $f \in \overline{\V}^*$, we sometimes abuse notation and write $f(1 \otimes v)$ as $f(v)$.

\begin{prop}\label{field}
Let $n\geq 2$ be an integer, let $\V$ be a vector space over a field $\field$, and let ${\mathbb{K}\supseteq \field}$ be a field extension. A set of non-zero vectors $\{v_1,\dots, v_n\}\subseteq \V\setminus\{0\}$ splits if and only if ${\{1\otimes v_1,\dots, 1 \otimes v_n\}\subseteq \mathbb{K} \otimes_{\field} \V \setminus \{0\}}$ splits.
\end{prop}
\begin{proof}
The statement follows easily from the fact that linear dependence does not depend on field extensions.
\end{proof}

\begin{prop}\label{prop:min_split}
Let $\V$ be a vector space over a field $\field$, and let $v_1,\dots,v_n \in \V\setminus\{0\}$ be non-zero vectors. If $\{v_1,\dots, v_n\}$ splits, then it is not minimal. If $\field$ is infinite, then $\{v_1,\dots, v_n\}$ splits if and only if it is not minimal.
\end{prop}
The converse does not always hold over a finite field. The set $\{1,1,1\}$ is not minimal over $\mathbb{Z}_2$ (as a vector space over itself), and also does not split. Another difference between these notions is that splitting is invariant under field extensions, whereas minimality is not (as evidenced by this example).
\begin{proof}[Proof of Proposition~\ref{prop:min_split}]
Suppose $\sum_{a \in [n]} \alpha_a v_a=0$ constitutes a minimal linear dependence of $\{v_1,\dots, v_n\}$. Then $\{v_1,\dots, v_n\}$ clearly does not split, since for any partition $S \sqcup T=[n]$,
\begin{align}
\sum_{a \in S} \alpha_a v_a \in {\spn\{v_a : a \in S\} \cap \spn\{v_a : a \in T\}}
\end{align}
is non-zero. It remains to show that if $\field$ is infinite and $\{v_1,\dots, v_n\}$ is not minimal, then it splits. We use basic algebraic geometry for this argument, for which we refer the reader to~\cite{harris2013algebraic}. Let
\begin{align}
X&=\{(u_1,\dots, u_n)\in \V \times \dots \times \V : u_a \in \spn\{v_a\} \text{ for all } a \in [n]\},\\
Y&=\{(u_1,\dots, u_n) \in \V \times \dots \times \V : \sum_{a\in [n]} u_a =0\},\\
U_S&=\{(u_1,\dots,u_n) \in X \cap Y : \sum_{a \in S} u_a \neq 0\} \quad \text{for each} \quad S \subseteq [n],\\
U&=\bigcap_{\substack{S\subseteq [n]\\ 1 \leq \abs{S}\leq n-1}}U_S.
\end{align}
Since $X \cap Y$ is a linear subspace and $\field$ is infinite, $X \cap Y$ is an irreducible algebraic variety. Note that each $U_S$ is Zariski open in $X \cap Y$. Since $\{v_1,\dots, v_n\}$ is not minimal, $U=\{\}$, which implies that $U_S = \{\}$ for some $S \in [n]$ with $1 \leq \abs{S} \leq n-1$. This implies $U_{[n]\setminus S}=\{\}$, so for every $u \in \spn\{ v_a : a \in S\}$ and $w \in \spn\{v_a : a \in [n]\setminus S\}$ such that $u+w=0$, it holds that $u=w=0$. It follows that $\{v_1,\dots, v_n\}$ splits with respect to the partition $S\sqcup ([n]\setminus S)=[n]$.
\end{proof}

%-----------------------------------------------------------------------
\section{Applications of the main conjecture to uniqueness of tensor rank decompositions}\label{uniqueness_applications}
%-----------------------------------------------------------------------

In this section we prove that Conjecture~\ref{conjecture} would imply Conjecture~\ref{k-gen}, a uniqueness criterion that would generalize Kruskal's theorem. We then compare Conjecture~\ref{k-gen} to uniqueness criteria obtained by Domanov, De Lathauwer, and S\o{}rensen (DLS) in the case of three subsystems \cite{domanov2013uniqueness,domanov2013uniqueness2,domanov2014canonical,sorensen2015new,sorensen2015coupled}. These are the most general known uniqueness criteria that we are aware of, apart from the case when $n$ is small \cite{chiantini,unknown} or $k_1=d_1=n$ \cite{doi:10.1137/090779632}.

\begin{proof}[Proof that Conjecture~\ref{conjecture} would imply Conjecture~\ref{k-gen}]
Let $x_a=x_{a,1}\otimes \dots \otimes x_{a,m}$ for each $a \in [n]$. Suppose that $\;2 \abs{S} \leq\sum_{j=1}^m (d_j^S-1)+1$ whenever $2 \leq \abs{S} \leq n$, and $\sum_{a \in [n]} x_a = \sum_{a \in [n]} y_a$ for some product tensors $y_a=y_{a,1}\otimes \dots \otimes y_{a,m}$. It suffices to show that there exists a permutation $\sigma \in S_n$ for which $x_a=y_{\sigma(a)}$ for all $a \in [n]$. Let $R=\{x_1,\dots, x_n\}$ and ${Q=\{-y_1,\dots, -y_n\}}$. As with any set of non-zero vectors in a vector space, there exists a positive integer $t$ and a (unique) collection of disjoint, non-empty sets
\begin{align}
T_1,\dots,T_t \subseteq R \cup Q
\end{align}
such that $T_i$ does not split for all $i \in [t]$,
\begin{align}
 T_1 \cup \dots \cup T_t = R\cup Q,
\end{align}
and
\begin{align}
\spn(R\cup Q)=\bigoplus_{i\in [t]} \spn(T_i).
\end{align}
Since $\sum_{a \in [n]} x_a-\sum_{a \in [n]} y_a=0$, this implies that the elements of $T_i$ sum to zero for all $i \in [t]$.

If $\abs{T_i}=2$ for all $i \in [t]$, then each set $T_i$ contains one element of $R$ and one element of $Q$, since no two elements of $\{x_1,\dots, x_n\}$ sum to zero. This will complete the proof, as it shows that $x_a=y_{\sigma(a)}$ for all $a \in [n]$, where $\sigma\in S_n$ is chosen so that $x_a$ and $-y_{\sigma(a)}$ lie in the same two-element set. Suppose toward contradiction that not every set $T_i$ has size two. Then there exists an index $i \in [t]$ such that $\abs{ T_i \cap R } \geq 2$ and
\begin{align}\label{greater_intersection}
\abs{ T_i \cap R } \geq \abs{ T_i \cap Q}.
\end{align}
We fix such an index $i$ for the remainder of the proof. Note that
\begin{align}
\abs{T_i} \leq 2 \abs{T_i \cap R} \leq \sum_{j=1}^m (d_j^{T_i\cap R}-1)+1 \leq \sum_{j=1}^m (d_j^{T_i}-1)+1,
\end{align}
where the first inequality follows from~\eqref{greater_intersection}, the second is by assumption, and the third is trivial. By Conjecture~\ref{conjecture}, $T_i$ splits, a contradiction.
\end{proof}

%In the remainder of this section, we observe that whether or not a given decomposition constitutes a unique tensor rank decomposition is theoretically computable, but likely computationally intractable. We then combine the efficiently checkable uniqueness conditions of DLS into a single statement in Theorem~\ref{uniqueness}, and generalize one of them. We observe generalizations of some of these results that would follow from our Conjecture~\ref{k-gen}, and give evidence for a generalization of Theorem~\ref{uniqueness}.

Now we compare Conjecture~\ref{k-gen} to the uniqueness criteria of DLS. For a set of product tensors,
\begin{align}
\{x_a : a \in [n]\} \subseteq \pro{\V_1:\V_2 : \V_3},
\end{align}
let $k_j=\krank(x_{a,j} : a \in [n])$ for each $j \in [3]$. For each subset $S \subseteq [n]$ of size $2 \leq \abs{S} \leq n$, let
\begin{align}
d_j^S=\dim\spn\{x_{a,j} : a \in [n]\}\quad\text{for all}\quad j\in [3].
\end{align}
We use the shorthand $d_j=d_j^{[n]}$. As we will see, all of the uniqueness criteria of DLS with a known efficient implementation require the following condition to hold:
\begin{align}\label{k-rank-large}
&\min\{k_2,k_3\} \geq n-d_1+2,\nonumber\\
\text{ or } &\min\{k_1,k_3\} \geq n-d_2 + 2,\nonumber\\
\text{ or } & \min\{k_1,k_2\} \geq n-d_3 +2.
%\begin{cases}
%\min\{k_2,k_3\} \geq n-d_1+2\\
%\min\{k_1,k_3\} \geq n-d_2 + 2\\
%\min\{k_1,k_2\} \geq n-d_3 +2
%\end{cases}
\end{align}
This is a major drawback, as it means that these results cannot certify uniqueness when the k-ranks are small. For example, if $k_2= k_3=2$, then these results can only certify uniqueness if $d_1=n$. The following example shows that Conjecture~\ref{k-gen} does not require~\eqref{k-rank-large} to hold.

\begin{example}[See example 5.2 in \cite{domanov2013uniqueness2}]\label{uniqueness_example}
\begin{align}\{(&\alpha_1 e_1)\otimes e_1 \otimes e_1,(\alpha_2 e_2)\otimes e_2 \otimes e_2, (\alpha_3 e_3) \otimes e_3\otimes e_3,\\
&(\alpha_4 e_4)\otimes e_4\otimes e_4,(\alpha_5(e_2+e_3))\otimes (e_2+e_4)\otimes (e_1+e_4)\} \quad \text{for} \quad \alpha_1,\dots, \alpha_5 \in \field^{\times}.
\end{align}

In this example, $k_1=k_2=k_3=2$, $d_1=d_2=d_3=4$, and $n-d_j+2=3$ for all $j \in [3]$, so~\eqref{k-rank-large} does not hold.
%In fact, an even less restrictive condition, Condition~W below, does not hold, so even the difficult-to-check uniqueness conditions proposed in~\cite{domanov2013uniqueness2} involving Condition~W do not apply. 
In~\cite{domanov2013uniqueness2}, it is proven that if $\alpha_2=\dots=\alpha_5=1$, then the sum of these product tensors constitutes a unique tensor rank decomposition.
%(This is done in order to show that Condition~W is not necessary for uniqueness.)
The proof given in~\cite{domanov2013uniqueness2} is quite complicated and specific to this case. This is to be expected, as uniqueness does not follow directly from any known efficiently-checkable uniqueness criteria. Uniqueness for arbitrary $\alpha_1,\dots, \alpha_5 \in \field^{\times}$ would follow easily from Conjecture~\ref{k-gen}.
\end{example}

To combine and analyze the uniqueness criteria of DLS, we recall Conditions~U,~H, and~C from \cite{domanov2013uniqueness,domanov2013uniqueness2}, which will be combined with other conditions to certify uniqueness. For notational convenience, we have changed these definitions slightly from \cite{domanov2013uniqueness,domanov2013uniqueness2}. For example, our Condition~U is their Condition~$U_{n-d_1+2}$, with the added condition that $k_1 \geq 2$. For a vector $\alpha \in \field^n$, we let $\omega(\alpha)$ denote the number of non-zero entries in $\alpha$.

%
%Note that the question of whether $\sum_{a \in [n]} x_a$ constitutes a unique tensor rank decomposition is theoretically computable, as it can be formulated as the following ideal membership problem: Is the (Zariski closed) set
%\begin{align}
%\{(y_1,\dots, y_n)\in \pro{\V_1:\V_2 : \V_3}^{\times n} : \sum_a y_a=\sum_a x_a\} \subseteq \V^{\times n}
%\end{align}
%contained in the Zariski closed set
%\begin{align}
%\bigcup_{\sigma \in S_n}\{(y_1,\dots, y_n)\in \pro{\V_1:\V_2 : \V_3}^{\times n} : y_{a}=x_{\sigma(a)} \quad \text{for all} \quad a \in [n]\} \subseteq \V^{\times n}?
%\end{align}
%Such questions can theoretically be answered deterministically using Gröbner bases, however these algorithms often require exponential space in the number of variables (even to certify membership of a single polynomial in an ideal). The number of variables in this problem is $n(d_1+d_2+d_3)$, so one would strongly expect this problem to quickly become intractable for large $n$.

%\begin{condition}[See Equation~1.12 in \cite{domanov2013uniqueness2}]\label{1.12}
%There exists $\sigma \in S_3$ such that
%\begin{align}\label{1.12eq}
%k_{\sigma(1)}+\max\{ \min\{k_{\sigma(2)},k_{\sigma(3)}-1\},\min\{k_{\sigma(2)}-1,k_{\sigma(3)}\}\}\geq n+1.
%\end{align}
%\end{condition}

%\begin{namedtheorem}{Condition W}
%It holds that $k_{1}\geq 2$, and for all $f \in \V_{1}^*$,
%\begin{align}
%\setft{rank}\Big[\sum_{a \in [n]} f(x_{a,1}) x_{a,2} \otimes x_{a,3}\Big]\geq n-d_1+2 \quad \text{whenever}\quad \omega(f(x_{1,1}),\dots,f(x_{n,1})) \geq n-d_{1}+2.
%\end{align}
%\end{namedtheorem}

\begin{namedtheorem}{Condition U}
It holds that $k_{1} \geq 2$, and for all $\alpha \in \field^n$,
\begin{align}\label{Umeq}
\rank\Big[\sum_{a \in [n]} \alpha_a x_{a, 2}\otimes x_{a,3} \Big]\geq n-d_1+2 \quad\text{whenever}\quad \omega(\alpha) \geq n-d_{1}+2.
\end{align}
\end{namedtheorem}
A less-restrictive condition than Condition~U, which we would call Condition~W, also appears in \cite{domanov2013uniqueness,domanov2013uniqueness2}, and is the same as Condition~U except that it only requires~\eqref{Umeq} to hold when ${\alpha = (f(x_{1,1}),\dots, f(x_{n,1}))}$ for some linear functional $f \in \V_1^*$. We are not aware of an efficient method to check either Condition~U or Condition~W. Nevertheless, we have included Condition~U because it will help us form a better theoretical picture of the uniqueness criteria of DSL.

Now we state Conditions~H and~C, which are more restrictive than Condition~U.

\begin{namedtheorem}{Condition H}
It holds that $k_{1} \geq 2$, and
\begin{align}\label{k-gen2eq}
d_{2}^S\!+d_{3}^S\!-\abs{S} \geq \min\{\abs{S}, n-d_{1}\!+2\} \! \quad \! \text{for all} \! \! \quad \! S \subseteq [n] \! \quad \!\! \text{of size} \quad \!\!\! 2 \leq \abs{S}\leq n.
\end{align}
\end{namedtheorem}
%Note that Condition~\ref{Wm} is less restrictive than Condition~\ref{Um}. Proposition 1.22 in~\cite{domanov2013uniqueness2} states that Conditions~\ref{1.12} and~\ref{Wm} imply uniqueness. Corollary 1.23 in~\cite{domanov2013uniqueness2} states that Conditions~\ref{1.12} and~\ref{Um} imply uniqueness. Proposition~1.26 in~\cite{domanov2013uniqueness2} states that if Condition~\ref{Um} holds for two distinct permutations $\sigma, \mu \in S_3$, then the decomposition is unique. These results are some of the most general known sufficient conditions for uniqueness, but unfortunately Conditions~\ref{Wm} and~\ref{Um} can be difficult to check. As a result, there are two more restrictive but easily checkable conditions that are often used instead. The first is easy to describe.
Condition~C takes a bit more work to describe. For positive integers $q,r,$ and $t$, and matrices
\begin{align}
Y=(y_1,\dots, y_t) \in \Lin(\field^t, \field^q)\\
Z=(z_1,\dots, z_t)\in \Lin(\field^t, \field^r),
\end{align}
let
\begin{align}
Y \odot Z=(y_1\otimes z_1, \dots, y_t \otimes z_t) \in \Lin(\field^t,\field^{qr})
\end{align}
denote the \textit{Khatri-Rao product} of $Y$ and $Z$. Suppose $\V_j=\field^{d_j}$ for each $j \in [3]$, and consider the matrices
\begin{align}
X_j=(x_{1,j},\dots, x_{n,j})\in \Lin(\field^n,\field^{d_j}) \quad \text{for} \quad j\in[3].
\end{align}
For a positive integer $s\leq d_j$, let $\C_s(X_j)$ be the $\binom{d_j}{s} \times \binom{n}{s}$ matrix of $s \times s$ minors of $X_j$, with rows and columns arranged according to the lexicographic order on the size $s$ subsets of $[d_j]$ and $[n]$, respectively. Define the matrix
\begin{align}
C_s=\C_s(X_{2})\odot \C_s(X_{3}) \in \Lin(\field^{\binom{n}{s}}, \field^{q}),
\end{align}
where $q=\big(\substack{d_{2}\\s}\big)\big(\substack{d_{3}\\s}\big)$. Now we can state the next condition.
%It turns out that Condition~W is equivalent to the existence of a permutation $\sigma \in S_3$ such that 
%\begin{align}
%\C_{s}^\sigma v \neq 0
%\end{align}
%for $s=n-d_{\sigma(1)}+2$ and a certain collection of non-zero vectors $v \in \field^{\binom{n}{l}}$. Condition~U is equivalent to the same statement, but for a larger collection of vectors. Therefore, Conditions~\ref{Wm} and~\ref{Um} are both implied by the following easily checkable condition.
\begin{namedtheorem}{Condition C}
It holds that $k_{1} \geq 2$, $\min\{d_2,d_3\} \geq n-d_1+2$, and
\begin{align}
\rank(C_{n-d_{1}+2})=\left(\substack{ n \\ n-d_{1}+2} \right).
\end{align}
\end{namedtheorem}
Now we state the condition of our Conjecture~\ref{k-gen} in the case of three subsystems. Unlike Conditions~U,~H, and~C, the following condition does not appear in~\cite{domanov2013uniqueness,domanov2013uniqueness2}, nor anywhere else that we are aware of.
\begin{namedtheorem}{Condition S}
It holds that
\begin{align}\label{concon}
2 \abs{S} \leq d_1^S+d_2^S+d_3^S-2 \quad \text{for all} \quad S \subseteq [n] \quad \text{of size} \quad 2 \leq \abs{S}\leq n.
\end{align}
\end{namedtheorem}

The following implications hold:
%\begin{equation}
%\begin{tikzcd}[row sep=tiny]
%\text{Condition~H} \arrow[Rightarrow,dr] & & &\\
% & \hspace{-5pt} & \text{Condition~U} \arrow[Rightarrow,r] & \begin{cases} \text{Condition~W}\\ \min\{k_2,k_3\} \geq n-d_1+2\\ \dim\spn \{ x_{a,\hat{1}} : a \in [n]\}=n\end{cases} \\
%\text{Condition~C} \arrow[Rightarrow,ur] & & &
%\end{tikzcd}
%\end{equation}
%
%\begin{equation}
%\begin{tikzcd}[row sep=tiny]
%\text{Condition~H} \arrow[Rightarrow,dr]  & &\\
% &  \hspace{1pt} \text{Condition~U} \arrow[Rightarrow,r] & \begin{cases} \text{Condition~W}\\ \min\{k_2,k_3\} \geq n-d_1+2\\ \dim\spn \{ x_{a,\hat{1}} : a \in [n]\}=n\end{cases} \\
%\text{Condition~C} \arrow[Rightarrow,ur] & &
%\end{tikzcd}
%\end{equation}

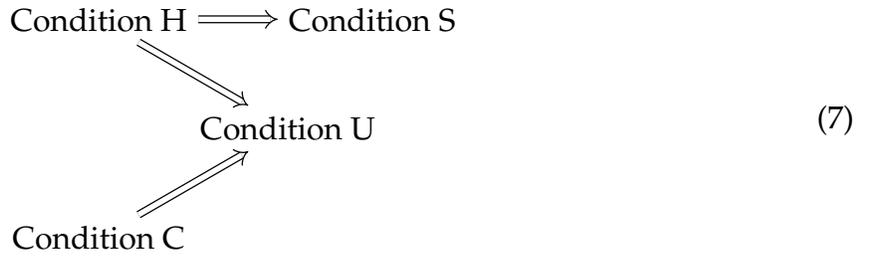
\begin{equation}\label{implications}
\begin{tikzpicture}[commutative diagrams/every diagram]
\node(dummy){};
\node[above=of dummy](H){\text{Condition H}};
\node[below=of dummy](C){\text{Condition C}};
\node[right=30pt of dummy](U){\text{Condition U}};
\node[right=30pt of H](S){\text{Condition S}};
%\node[right=30pt of U](W){$\begin{cases} \text{Condition~W}\\ \min\{k_2,k_3\} \geq n-d_1+2\\ \dim\spn \{ x_{a,\hat{1}} : a \in [n]\}=n.\end{cases}$};
\path[commutative diagrams/.cd, every arrow, every label]
(H) edge[commutative diagrams/Rightarrow]  (U)
(C) edge[commutative diagrams/Rightarrow] (U)
%(U) edge[commutative diagrams/Rightarrow] (W)
(H) edge[commutative diagrams/Rightarrow] (S);
\end{tikzpicture}
\end{equation}
In the case of three subsystems, our Conjecture~\ref{k-gen} states that Condition~S implies uniqueness. Since Condition~H implies Condition~S, then a corollary to Conjecture~\ref{k-gen} would be that Condition~H implies uniqueness.

All of the implications in~\eqref{implications} except (Condition H $\Rightarrow$ Condition S) were proven in~\cite{domanov2013uniqueness}. To see that Condition~H $\Rightarrow$ Condition S, note that for any subset $S\subseteq [n]$ of size $2 \leq \abs{S} \leq n$, the condition $k_1 \geq 2$ implies
\begin{align}
d_{1}^S \geq \max\{2,d_{1}-(n-\abs{S})\},
\end{align}
so by Condition~H,
\begin{align}
d_{1}^S+d_{2}^S+d_{3}^S &\geq \max\{2,d_{1}-(n-\abs{S})\} + \abs{S}+\min\{\abs{S},n-d_{1}+2\}\\
&\geq 2 \abs{S}+2,
\end{align}
and Condition S holds. It is easy to find examples that certify {Condition~C $\not\Rightarrow$ Condition~S}.  By Example~\ref{uniqueness_example}, {Condition~S $\not\Rightarrow$ Condition~U}. In~\cite{domanov2013uniqueness} it is asked whether Condition~H $\Rightarrow$ Condition~C. Condition~U is theoretically computable, as it can be phrased as an ideal membership problem, however we are unaware of an efficient implementation. Conditions~C,~H, and~S are clearly easy to check.

The following theorem contains every uniqueness criterion of DLS for which we are aware of an efficient implementation. This theorem is stated in terms of Condition~U to maintain generality, however, only the implied statements in which Condition~U is replaced by Condition~H or~C  (see~\eqref{implications}) have an efficient implementation that we are aware of.
%Our Conjecture~\ref{k-gen} would generalize the version of Theorem~\ref{uniqueness} with Condition~U replaced by Condition~H, to the statement that Condition~H alone certifies uniqueness.

%We have omitted a few statements involving Condition~W that appear in~\cite{domanov2013uniqueness,domanov2013uniqueness2}, as they are more complicated to state and even more difficult to check than the statements involving Condition~U.

{\begin{theorem}\label{uniqueness}
Let $n\geq 2$ be an integer, let $\V=\V_1\otimes\V_2\otimes\V_3$ be a tripartite vector space over a field $\field$, and let
\begin{align}
\{x_a : a \in [n]\} \subseteq \pro{\V_1 : \V_2 : \V_3}
\end{align}
be a set of product tensors. Suppose that Condition~U holds, and any one of the following conditions holds: 
\begin{enumerate}[align=left]
%\item \hfill $k_{\sigma(1)}+\max\{ \min\{k_{\sigma(2)},k_{\sigma(3)}-1\},\min\{k_{\sigma(2)}-1,k_{\sigma(3)}\}\}\geq n+1$ \hfill 
%\refstepcounter{equation}(\textup{\theequation})
\item \leavevmode\vspace{-\dimexpr\baselineskip + \topsep + \parsep} \begin{center}$k_{1}+\min\{k_{2},k_{3}-1\}\geq n+1.$ \end{center}
\item It holds that $k_2 \geq 2$ and for all $\alpha \in \field^n$,
\begin{align}
\setft{rank}\Big[\sum_{a \in [n]} \alpha_a x_{a,1}\otimes x_{a,3}\Big]\geq n-d_2+2 \quad\text{whenever}\quad \omega(\alpha) \geq n-d_{2}+2.
\end{align}
(Note that this is just Condition U with the first subsystem replaced by the second).
\item There exists a subset $S\subseteq [n]$ with $0 \leq |S| \leq d_{1}$ such that
\begin{enumerate}
\item \leavevmode\vspace{-\dimexpr\baselineskip + \topsep + \parsep} \begin{center} $d_{1}^S=|S|.$ \end{center}
\item \leavevmode\vspace{-\dimexpr\baselineskip + \topsep + \parsep} \begin{center} $d_{2}^{[n]\setminus S}=n-|S|.$\end{center}
\item For any linear map $\Pi \in \Lin(\V_{1})$ with ${\ker(\Pi) = \spn \{x_{a ,1}: a \in S\}}$, scalars $\alpha_1,\dots, \alpha_n \in \field$, and index $b \in [n]\setminus S$ such that
\begin{align}
\sum_{a \in [n]\setminus S} \alpha_a & \Pi x_{a, 1} \otimes x_{a, 3} = \Pi x_{b, 1} \otimes z \quad \text{ for some } \quad z \in \V_{\sigma(3)},
\end{align}
it holds that $\omega(\alpha)\leq 1$.
\end{enumerate}

\item There exists a permutation $\tau \in S_n$ for which the matrix
\begin{align}
X_{1}^\tau = (x_{\tau(1),1},\dots, x_{\tau(n),n})
\end{align}
has reduced row echelon form
\begin{align}\label{Y}
Y=\left[\begin{array}{@{}c | c@{}}
  \begin{matrix}
  1 & & \\
   & \ddots &\\
   && 1
  \end{matrix}
  & \begin{matrix}
   & & \\
   & {\normalfont \Huge Z} &\\
   && 
  \end{matrix}
%  \\
%  \hline
%   &
\end{array}\right],
\end{align}
where $Z \in \Lin(\field^{n-d_{1}},\field^{d_{1}})$ and the blank entries are zero. Furthermore, for each $a \in [d_{1}-1]$, the columns of the submatrix of $Y$ with row index $\{a,a+1,\dots, d_{1}\}$ and column index $\{a, a+1,\dots, n\}$ have k-rank at least two.
\item \leavevmode\vspace{-\dimexpr\baselineskip + \topsep + \parsep} \begin{center} $k_{1}=d_{1}.$ \end{center}
\item It holds that $k_1 \geq 2$ and for all $\alpha \in \field^n$,
\begin{align}
\setft{rank}\Big[\sum_{a \in [n]} \alpha_a x_{a,2}\otimes x_{a,3}\Big]\geq n-k_1+2 \quad\text{whenever}\quad \omega(\alpha) 
\geq n-k_{1}+2.
\end{align}
(Note that this is a stronger statement than Condition U, as it replaces the quantity ${n-d_1+2}$ with the possibly larger quantity $n-k_1+2$.)
\end{enumerate}
Then $\sum_{a\in[n]} x_a$ is a unique tensor rank decomposition.
\end{theorem}}

For each $i \in [5]$, we will refer to Theorem~\hyperref[uniqueness]{\ref{uniqueness}.i} as the statement that Condition~U and the $i$-th condition appearing in Theorem~\ref{uniqueness} imply uniqueness. Theorems~\hyperref[uniqueness]{\ref{uniqueness}.1} and~\hyperref[uniqueness]{\ref{uniqueness}.2} are Corollary~1.23 and Proposition~1.26 in \cite{domanov2013uniqueness2,domanov2014canonical}. The Condition~C version of Theorem~\hyperref[uniqueness]{\ref{uniqueness}.3} is stated in Theorem~2.2 in \cite{sorensen2015coupled}, although the proof is contained in \cite{domanov2013uniqueness,domanov2013uniqueness2,sorensen2015new}. Condition 3b in Theorem~\ref{uniqueness} can be formulated as checking the rank of a certain matrix (see \cite{sorensen2015coupled}). Theorem~\hyperref[uniqueness]{\ref{uniqueness}.4} is a new result that we will prove. The Condition~C version of Theorems~\hyperref[uniqueness]{\ref{uniqueness}.5} and~\hyperref[uniqueness]{\ref{uniqueness}.6} are Theorems~1.6 and~1.7 in~\cite{domanov2014canonical}. It is easy to see that Theorem~\hyperref[uniqueness]{\ref{uniqueness}.4} implies Theorem~\hyperref[uniqueness]{\ref{uniqueness}.5}, which in turn implies Theorem~\hyperref[uniqueness]{\ref{uniqueness}.6} by the arguments used in~\cite{domanov2014canonical}.

Most of these statements have previously only been formulated for $\field=\real$ or $\field=\complex$, however in all of these cases the proof can be adapted to hold over an arbitrary field. The first step in proving all of these statements is to show that Condition~U implies uniqueness in the first subsystem. This is Proposition~4.3 in~\cite{domanov2013uniqueness}, and it is proven using Kruskal's permutation lemma \cite{kruskal1977three} (the proof of the permutation lemma in~\cite{landsberg2012tensors} holds word-for-word over an arbitrary field). In fact, uniqueness in the first subsystem holds even with the assumption $k_1 \geq 2$ removed from Condition~U~\cite{domanov2013uniqueness}.

Recall that Conjecture~\ref{k-gen} would imply that Condition~H alone certifies uniqueness, and would thus generalize the version of Theorem~\ref{uniqueness} with Condition~U replaced by Condition~H. A natural question that then arises is whether Condition~U alone certifies uniqueness. Theorems~\hyperref[uniqueness]{\ref{uniqueness}.4} and~\hyperref[uniqueness]{\ref{uniqueness}.5} are distinguished among the results in Theorem~\ref{uniqueness}, in that the extra conditions they impose beyond Condition~U concern only the first subsystem. One can view these results as further evidence that Condition~U alone certifies uniqueness, as they show that no further conditions on the second and third subsystems are necessary for uniqueness.

Now we prove Theorem~\hyperref[uniqueness]{\ref{uniqueness}.4}, for which we require the following proposition.

\begin{prop}
Condition 4 in Theorem~\ref{uniqueness} holds if and only if there exists a permutation $\tau \in S_n$ such that for each $a \in [d_1-1]$ there exists a linear operator $\Pi_a \in \Lin(\V_1)$ for which
\begin{align}
\Pi_a (x_{\tau(b),1})=0 \quad \text{for all}\quad b \in [a-1],
\end{align}
and
\begin{align}\label{k-rank-thingy}
\krank( \Pi_a x_{\tau(a),1},\dots,\Pi_a x_{\tau(n),1})\geq 2.
\end{align}
\end{prop}
\begin{proof}
Assume without loss of generality that $\V_1=\field^{d_1}$. To see that the first statement implies the second, for each $a\in [d_1-1]$ let $\Pi_a=D_a P$, where $P\in \Lin(\field^{d_1})$ is the invertible matrix for which $PX_1^\sigma=Y$, and $D_a \in \Lin(\field^{d_1})$ is the diagonal matrix with the first $a-1$ entries zero and the remaining entries $1$. It is easy to verify that~\eqref{k-rank-thingy} holds.

%\begin{align}\label{D}
%D_a=\left[\begin{array}{@{}c | c@{}}
%  \begin{matrix}
%  0 & & &&&\\
%   & \ddots &&&&\\
%   && 0&&&\\
%   &&&1&&\\
%   &&&&\ddots&\\
%   &&&&&1
%  \end{matrix}
%  & \begin{matrix}
% &&\\
% &&\\
% &&\\
% &&\\
% &&\\
% &&
%  \end{matrix}
%\end{array}\right],
%\end{align}
%where the lefthand block is a $d_1\times d_1$ diagonal matrix with the first $a-1$ entries zero and the remaining entries $1$, and the last $n-d_1$ columns zero.

Conversely, suppose that the reduced row echelon form of $X_1^\tau$, given by $P X_1^\tau$ for some invertible matrix $P \in \Lin(\field^{d_1})$, does not have the specified form. Then there exists ${a\in [d_1-1]}$ for which the columns of $D_a P X_1^\tau$ have k-rank at most one. Any matrix $\Pi_a\in \Lin(\field^{d_1})$ for which $\Pi_a (x_{\tau(b),1})=0$ for all $b \in [a-1]$ satisfies
\begin{align}
\Pi_a =\Pi_a P^{-1} D_a P.
\end{align}
Since the k-rank is non-increasing under matrix multiplication from the left,~\eqref{k-rank-thingy} does not hold.
\end{proof}

\begin{proof}
[Proof of Theorem~\ref{uniqueness}.4]
Whether the decomposition $\sum_{a\in [n]} x_a$ constitutes a unique tensor rank decomposition is invariant under permutations $\tau \in S_n$ of the tensors, so it suffices to prove the statement under the assumption that the permutation $\tau$ appearing in Condition 4 is trivial. We prove the statement by induction on $d_1$. If $d_1=2$, then Condition~U implies $k_2=k_3=n$, so uniqueness follows from Kruskal's theorem (or Theorem~\ref{tripartite_thm}). For $d_1>2$, suppose $\sum_{a\in [n]} x_a =\sum_{a\in [n]} y_a$ for some set of product tensors
\begin{align}
\{y_a : a \in [n]\} \subseteq \pro{\V_1: \V_2 : \V_3}.
\end{align}
By Proposition~4.3 in~\cite{domanov2013uniqueness} (or rather, the extension of this result to an arbitrary field), there exists a permutation $\sigma \in S_n$ and nonegative integers $\alpha_1,\dots, \alpha_n \in \field^\times$ such that $\alpha_a x_{a,1}=y_{\sigma(a),1}$ for all $a \in [n]$. Let $\Pi_1 \in \Lin(\V_1)$ be any operator for which ${\ker(\Pi_1)=\spn\{x_{a,1}\}}$ and~\eqref{k-rank-thingy} holds (recall that $\tau$ is trivial). Then
\begin{align}
\sum_{a \in [n] \setminus \{1\}} (\Pi_1 x_{a,1})\otimes x_{a,2} \otimes x_{a,3}=\sum_{a \in [n] \setminus \{1\}} (\alpha_{a} \Pi_1  x_{a,1})\otimes y_{\sigma(a),2} \otimes y_{\sigma(a),3}.
\end{align}
Now, $\dim\spn\{\Pi_1 x_{a,1}: a \in [n]\setminus\{1\}\}=d_1-1$, and Condition~U again holds for the set of product tensors 
\begin{align}
\{(\Pi_1 x_{a,1}) \otimes x_{a,2} \otimes x_{a,3} : a \in [n] \setminus \{1\}\}.
\end{align}
Furthermore, these product tensors again satisfy Condition 4 of Theorem~\ref{uniqueness}, so by the induction hypothesis
\begin{align}
(\Pi_1 x_{a,1}) \otimes x_{a,2} \otimes x_{a,3}=(\alpha_a \Pi_1 x_{a,1}) \otimes y_{\sigma(a),2} \otimes y_{\sigma(a),3}\quad \text{for all}\quad a \in [n]\setminus \{1\}.
\end{align}
It follows that $x_a=y_{\sigma(a)}$ for all $a \in [n] \setminus\{1\}$, so $x_1=y_{\sigma(1)}$, which completes the proof.
\end{proof}

We conclude this section by proving Theorem~\ref{connect}, a statement on the multilinear rank of linear combinations of product tensors with large k-ranks, which generalizes results in~\cite{1751-8121-48-4-045303}. We then use Theorem~\ref{connect}, along with a natural generalization of Condition~U to at least three subsystems, to show that if $2n \leq (d_1-1)+\sum_{j=2}^m (k_j-1)+1$, then the decomposition is unique in the first subsystem. We do not claim that this uniqueness result is new, but merely include it to demonstrate one application of Theorem~\ref{connect}.

\begin{theorem}\label{connect}
Let $n \geq 2$, $m \geq 2$, and $r\in \{0,1,\dots,n\}$ be integers, let $\V=\V_1\otimes \dots\otimes \V_m$ be a multipartite vector space over a field $\field$, and let $\{{x_a}: a \in [n] \} \subset \pro{\V_1 : \dots : \V_m}$ be a set of product tensors. For each $j\in [m]$, let
\begin{align}
k_j=\krank (x_{1,j},\dots, x_{n,j}).
\end{align}
If ${n+r\leq\sum_{j=1}^m (k_j-1)+1}$, then for any $\alpha \in \field^n$, $j \in [m]$, it holds that
\begin{align}\label{connect_eq}
\rank_j\big(\sum_{a \in [n]} \alpha_a x_a\big) \geq r+1 \quad \text{whenever} \quad \omega(\alpha) \geq r+1.
\end{align}
Furthermore, Kruskal's theorem implies that if $k_j \geq 2$ in at least three indices $j \in [m]$, then for any subset $ S \subseteq [n]$ with $\abs{S}=r$ and non-zero scalars $\{\alpha_a : a \in S\} \subseteq \field^{\times}$ it holds that $\sum_{a \in S} \alpha_a x_a$ constitutes a unique tensor rank decomposition.
\end{theorem}

In particular, Theorem~\ref{connect} states that if ${n\leq \sum_{j=1}^m (k_j-1)}+1$, then $\{x_1,\dots, x_n\}$ are linearly independent; and if ${n \leq \sum_{j=1}^m (k_j-1)}$, then every product tensor in $\spn\{x_1,\dots, x_n\}$ is a scalar multiple of one of the product tensors $x_1,\dots, x_n$ (these are Proposition 3.1 and Theorem 3.2 of \cite{1751-8121-48-4-045303}). Theorem~\ref{connect} can be viewed as a family of statements that interpolate between the results of \cite{1751-8121-48-4-045303} (the cases $r=0$ and $r=1$), and Kruskal's theorem (the case $r=n$).

\begin{proof}[Proof of Theorem~\ref{connect}]
We first use Kruskal's theorem to prove the second statement that $\sum_{a \in S} \alpha_a x_a$ constitutes a unique tensor rank decomposition. The cases $r=0$ and $r=1$ are trivial, so assume $r \geq 2$. Let $k_j^S=\krank(x_{a,j} : a \in S)$, and note that $k_j^S= \min\{k_j, r\}$ for each $j \in [m]$. It is straightforward to verify that
\begin{align}
2r-1 \leq \sum_{j=1}^m (k_j^S-1),
\end{align}
which completes the proof of the second statement by Kruskal's theorem (Theorem~\ref{kruskal}).

For the first statement, it suffices to consider the case $r\leq n-1$, and to prove that for any subset $S \subseteq [n]$ of size ${r+1 \leq \abs{S} \leq n}$, and index $j\in [m]$, it holds that $\rank_j\big(\sum_{a \in S} x_a\big) \geq r+1$ (the scalars $\alpha_a$ can be absorbed into the $x_a$). Let $k_{\hat{j}}^S=\krank(x_{a,\hat{j}}: a \in S)$, and note that
\begin{align}
k_{\hat{j}}^S\geq \min\Big\{\sum_{i \in [m]\setminus\{j\}} (k_i-1)+1, \abs{S}\Big\}
\end{align}
by Lemma~1 in~\cite{sidiropoulos2000uniqueness}. Similarly, $k_j^S=\krank(x_{a,j}: a \in S) = \min\{k_j, \abs{S}\}$.
It follows that $r+1 \leq k_j^S + k_{\hat{j}}^S-\abs{S}$, so by Sylvester's rank inequality \cite{horn2013matrix}, ${\rank_j\big(\sum_{a \in S} x_a\big) \geq r+1}$.
\end{proof}
Now we use Theorem~\ref{connect}, and a generalization of Condition~U to the case of at least three subsystems, to prove a sufficient condition for uniqueness in one subsystem. Proposition 4.3 in~\cite{domanov2013uniqueness} states that, in the case of three subsystems, Condition~U implies the decomposition is unique in the first subsystem (even without the condition $k_1 \geq 2$). It is straightforward to verify that this statement can be generalized to the case of at least three subsystems as follows.

\begin{prop}\label{Ugen}
Let $n \geq 2$ and $m \geq 3$ be integers, let ${\V=\V_1\otimes \cdots \otimes \V_m}$ be a multipartite vector space over a field $\field$, and let
\begin{align}
\{x_a : a \in [n] \}\subseteq \pro{\V_1: \dots : \V_m}
\end{align}
be a set of product tensors with $d_1=\dim\spn\{x_{a,1}: a\in[n]\}$.
If
\begin{align}\label{Ugeneq}
\rank(\sum_{a \in [n]} \alpha_a x_{a,\hat{1}}) \geq n-d_1+2\quad\text{whenever}\quad\omega(\alpha)\geq n-d_1+2,
\end{align}
%\begin{align}
%\rank(\sum_{a \in [n]} \alpha_a x_{a, 2}\otimes\dots \otimes x_{a,m}) \geq n-d_1+2\quad\text{whenever}\quad \omega(\alpha)\geq n-d_1+2,
%\end{align}
then the decomposition $\sum_{a\in[n]} x_a$ is unique in the first subsystem.
\end{prop}
Equation~\eqref{Ugeneq}, paired with the condition $k_1 \geq 2$, is a natural generalization of Condition~U to the case of at least three subsystems.

Now we use Theorem~\ref{connect} and Proposition~\ref{Ugen} to prove that if
\begin{align}
{2n \leq (d_1-1)+\sum_{j=2}^m (k_j-1)+1},
\end{align}
then the decomposition is unique in the first subsystem. For any tensor $v$ and subsystem index $j$, it holds that $\rank_j(v)\leq \rank(v)$. Thus, the case $r=n-d_1+1$ in Theorem~\ref{connect} combined with Proposition~\ref{Ugen} implies that the decomposition is unique in the first subsystem. Theorem~\ref{connect} is actually overkill for this statement, as it would suffice to prove that $\rank_j(\sum_{a\in [n]} \alpha_a x_{a,\hat{1}}) \geq n-d_1+2$ for a single index $j \in \{2,3,\dots, m\}$.

%-----------------------------------------------------------------------------%
\section{Corollaries to the main conjecture on tensor rank and linearly dependent sets of product tensors}\label{consequences}

In this section we prove several corollaries to Theorem~\ref{s_cases}, and mention more general statements that would follow from Conjecture~\ref{conjecture}. We observe applications of these results to quantum information theory, linear preserver problems, and uniqueness criteria.

The first corollary is an upper bound on the number of subsystems $j\in [m]$ for which a circuit of product tensors can have rank greater than one. Our bound improves a result of Ballico \cite{ballico2020linearly}, and is sharp (see Section~\ref{conjecture_kruskal}).

\begin{cor}\label{linincor}
\sloppy
Let $n$ and $m$ be positive integers, and let $\V=\V_1\otimes \dots\otimes \V_m$ be a multipartite vector space over a field $\field$. If a set of product tensors $\{{x_a}: a \in [n] \}\subset \pro{\V_1 : \dots : \V_m}$ forms a circuit, then ${\dim\spn\{ {x_{a, j}}: a \in [n]\} >1}$ for at most $n-2$ indices $j \in [m]$.
\end{cor}
\begin{proof}
The result follows immediately from Theorem~\ref{sep_combo}, since circuits do not split.
\end{proof}

The next corollary follows immediately from Corollary~\ref{linincor}, and was used in \cite{westwick1967, johnston2011characterizing} to characterize the invertible linear operators in $\Lin(\V)$ preserving $\pro{\V_1 : \dots : \V_m}$. It would be interesting to see whether our more general results could be used to characterize preservers of tensor rank $r \geq 2$.
\begin{cor}[\cite{westwick1967, johnston2011characterizing}]\label{l1} Let $m \geq 1$ be an integer, let $\V=\V_1\otimes \dots\otimes \V_m$ be a multipartite vector space over a field $\field$, and let $x_1, x_2 \in \pro{\V_1 : \dots : \V_m}$ be product tensors. Then the following statements are equivalent:
\begin{enumerate}
\item There exists at most a single index $j \in [m]$ for which $\dim\spn \{{x_{1,j}}, x_{2,j}\}=2$.
\item For some non-zero scalars $\alpha_1, \alpha_2 \in \field^\times$, it holds that\\${\alpha_1{x_1}+\alpha_2 {x_2}\in \pro{\V_1 : \dots : \V_m}\cup\{0\}}$.
\item For all scalars $\alpha_1, \alpha_2 \in \field$, it holds that $\alpha_1{x_1}+\alpha_2 {x_2}\in \pro{\V_1 : \dots : \V_m} \cup \{0\}$.
\end{enumerate}
\end{cor}

Now we use Corollary~\ref{l1} to provide an alternate proof of one of the main mathematical results in~\cite{PhysRevA.95.032308}, which classifies two-dimensional subspaces of multipartite space according to how many one-dimensional subspaces they contain that consist entirely of product tensors. This result is interpreted in~\cite{PhysRevA.95.032308} in the context of quantum information theory as a classification of entanglement in rank-two density matrices, by identifying a density matrix with its eigenspace.

\begin{cor}[Theorem 11 in \cite{PhysRevA.95.032308}]\label{geo}
Let $m \geq 2$ be an integer and let $\V=\V_1\otimes \dots\otimes \V_m$ be a multipartite vector space over a field $\field$. Then every two-dimensional subspace $\W \subseteq \V$ falls into one of the following four categories.
\begin{enumerate}
\item $\W \subseteq \pro{\V_1 : \dots : \V_m} \cup \{0\}$.
\item \sloppy There exist precisely two distinct one-dimensional subspaces of $\W$ contained in $\pro{\V_1 : \dots : \V_m} \cup \{0\}$, and every other tensor in $\W$ is non-product.
\item \sloppy There exists precisely one one-dimensional subspace of $\W$ contained in $\pro{\V_1 : \dots : \V_m}\cup \{0\}$, and every other tensor in $\W$ is non-product.
\item Every non-zero tensor in $\W$ is non-product.
\end{enumerate}
\end{cor}

\begin{proof}
If every non-zero tensor in $\W$ is non-product, then $\W$ lies in the fourth category. If there exists precisely one one-dimensional subspace of $\W$ contained in $\pro{\V_1 : \dots : \V_m}\cup \{0\}$, then $\W$ lies in the third category. If there exist two distinct one-dimensional subspaces of $\W$ contained in $\pro{\V_1 : \dots : \V_m}\cup\{0\}$, then let $x_1, x_2$ be non-zero tensors contained in the first and second subspace respectively, so ${\W= \spn\{x_1,x_2\}}$. If there exists more than one index $j \in [m]$ for which ${\dim\spn\{x_{1,j}, x_{2,j}\}>1}$, then $\W$ lies in the second category by Corollary~\ref{l1}. If there exists one index $j \in [m]$ for which $\dim\spn\{x_{1,j}, x_{2,j}\}>1$, then $\W$ lies in the first category by Corollary~\ref{l1}.
\end{proof}

We next observe a consequence of Conjecture~\ref{conjecture} on the tensor rank of linear combinations of product tensors that would generalize Corollary~\ref{linincor}, and observe a connection between this result and Condition~U in the study of uniqueness criteria.

%a corollary to Theorem~\ref{sep_combo} with connections to Condition~U in the study of uniqueness conditions (see Section~\ref{uniqueness_applications}), and also to the generalization of Condition~U to the case of more than three subsystems (see~\eqref{Ugeneq}).

\begin{cor}\label{conjecture1}
The following statement is a corollary to Theorem~\ref{s_cases} in the restricted tripartite and restricted multipartite cases,
%when $m=3$, $d_1,d_2\geq 1$, and $1 \leq d_3\leq2$; or when $m$ is arbitrary, $d_1 \geq 1$, and ${1\leq d_2, \dots, d_m \leq 2}$.
and is merely a (conjectural) corollary to Conjecture~\ref{conjecture} in all other cases.

Let $n \geq 2$, $m \geq 3$, and ${r \in \{1,\dots, n-1\}}$ be integers, let $\V=\V_1\otimes\dots\otimes \V_m$ be a multipartite vector space over a field $\field$, and let ${\{{x_a}: a \in [n] \} \subset \pro{\V_1 : \dots : \V_m}}$ be a set of product tensors. For each $j \in [m]$, let
\begin{align}
d_j=\dim\spn\{{x_{a,j}} : a \in [n] \}.
\end{align}
If $n+r \leq \sum_{j=1}^m (d_j-1)+1$ and $\rank(\sum_{a \in [n]} x_a)\leq r$, then there exists a subset $S \subseteq [n]$ of size $r \leq \abs{S}\leq n$ for which
\begin{align}\label{Ulike}
\rank(\sum_{a \in S} x_a) < r.
\end{align}
\end{cor}

Note that the restricted multipartite case of Corollary~\ref{conjecture1} with $r=1$ is essentially Corollary~\ref{linincor}. In the bipartite case, it follows from Theorem~\ref{bipartite_prop} and similar arguments as the proof of Corollary~\ref{conjecture1} below that $n+r \leq (d_1-1)+(d_2-1)+1$ implies ${\rank(\sum_{a \in [n]} x_a)> r}$.

Corollary~\ref{conjecture1} would imply that if a set of product tensors
\begin{align}
{\{{x_a}: a \in [n] \} \subset \pro{\V_1 : \dots : \V_m}}
\end{align}
satisfies $d_1\geq 3$, $2n\leq \sum_{j=1}^m (d_j-1)$, and $\rank(\sum_{a\in[n]} \alpha_a x_{a,\hat{1}})=n-d_1+2$ for some non-zero scalars $\alpha_1,\dots,\alpha_n \in \field^{\times}$, then the multipartite generalization of Condition~U (see Equation \eqref{Ugeneq}) does not hold. We are not sure how useful this statement would be, as the condition $\rank(\sum_{a\in[n]} \alpha_a x_{a,\hat{1}})=n-d_1+2$ is quite specific.

\begin{proof}[Proof of Corollary~\ref{conjecture1}]
Since $\rank(\sum_{a\in [n]} x_a)\leq r$, there exist product tensors
\begin{align}
\{x_{n+1},\dots,x_{n+r}\} \subseteq \pro{\V_1 : \dots : \V_m}
\end{align}
for which
\begin{align}\label{k_sum_zero}
\sum_{a \in [n+r]} x_a =0.
\end{align}
Since $n+r \leq \sum_{j=1}^m (d_j-1)+1$, then by the non-minimal version of Conjecture~\ref{conjecture} (or Theorem~\ref{s_cases} in the special cases), there exists a subset $R \subseteq [n+r]$ of size ${1 \leq \abs{R} \leq n+r-1}$ such that
\begin{align}\label{k_subsum_zero}
\sum_{a \in  R} x_a= 0.
\end{align}
Define
\begin{align}
 S&= R \cap [n],\\
 T&= R \cap \{n+1, \dots, n+r\}.
\end{align}
We first consider the case $\abs{ S}>\abs{T}$. If $\abs{ S}\geq r$ we are done. If $\abs{ S}\leq r-1$, then for any subset $Q \subseteq [n]$ of size $\abs{Q}\geq r$ with $S\subseteq Q$, we have
\begin{align}
\sum_{a \in  Q} x_a= \bigg[\sum_{a \in  Q \setminus  S}  x_a\bigg]-\bigg[\sum_{b \in  T} x_b\bigg],
\end{align}
so $\rank(\sum_{a \in  Q} x_a)<r$.

Now we consider the case $\abs{S}\leq\abs{T}$. We have 
\begin{align}
\bigabs{[n]\setminus S}&=n-\abs{S} \\
&\geq r-\abs{T} + (n-r)\\
& > r-\abs{T}\\
&=\bigabs{[r]\setminus  T},
\end{align}
where the strict inequality follows from $n-r \geq 1$. But equations~\eqref{k_sum_zero} and~\eqref{k_subsum_zero} imply
\begin{align}
\sum_{a \in [n+r] \setminus  R} x_a=0.
\end{align}
The statement follows from the previous arguments with $R$ replaced by ${[n+r]\setminus  R}$.
\end{proof}

%Now we prove a corollary to Theorem~\ref{sep_combo} that gives a condition under which any two tensor rank decompositions agree on a subset. We mention a more general result that would follow from Conjecture~\ref{conjecture}.

Now we observe a consequence of Conjecture~\ref{conjecture} that gives a condition under which any two tensor rank decompositions agree on a subset.

\begin{cor}\label{ana_cor}
The following statement is a corollary to Theorem~\ref{s_cases} when $m=3$, $r_1,r_2\geq 1$, and $1 \leq r_3\leq2$; or when $m$ is arbitrary, $r_1 \geq 1$, and $1\leq r_2, \dots, r_m \leq 2$. This statement is merely a (conjectural) corollary to Conjecture~\ref{conjecture} in all other cases.

Let $n \geq 2$ and $m \geq 3$ be integers, let $\V=\V_1\otimes \dots\otimes \V_m$ be a multipartite vector space over a field $\field$, and let $v \in \V$ be a tensor of rank $n$ and multilinear rank $(r_1,\dots, r_m)$. If %For each $j \in [m]$, let $d_j$ be a positive integer no less than the rank of $v$ when $v$ is viewed as an element of $\Lin(\V_j^*,V_{\hat{j}})$ (see Section~\ref{mp}).
${2n \leq \sum_{j=1}^m (r_j-1)+1}$, then for any two sets of product tensors
\begin{align}
\{x_a : a \in [n]\} &\subseteq \pro{\V_1:\cdots:\V_m},\\
\{y_a : a \in [n]\} &\subseteq \pro{\V_1:\cdots:\V_m}
\end{align}
for which
\begin{align}
v=\sum_{a\in [n]} x_a=\sum_{a\in[n]} y_a,
\end{align}
there exists a permutation $\sigma \in \setft{S}_n$ and a subset ${S \subset [n]}$ of size $1 \leq \abs{S} \leq n-1$ such that
\begin{align}\label{subset_mult}
\sum_{a \in S} x_a = \sum_{a \in S} y_{\sigma(a)}.
\end{align}
\end{cor}

As a simple example, consider the product tensors
\begin{align}\label{simple_example}
x_1&=e_1\otimes e_1 \otimes e_1\\
x_2&=e_2 \otimes e_2 \otimes e_1\\
x_3&=e_3 \otimes e_3 \otimes e_2,
\end{align}
and let $v=x_1+x_2+x_3$. Then $\rank(v)=3$, and Corollary~\ref{ana_cor} verifies that for any other tensor rank decomposition $v=y_1+y_2+y_3$, there exists $a,b \in [3]$ such that $x_a=y_b$.

Corollary~\ref{ana_cor} can be compared to Kruskal's theorem, which gives sufficient conditions for~\eqref{subset_mult} to hold for every singleton $S=\{a\}\subseteq [n]$.

%However, our result is of a somewhat different flavour than Kruskal's theorem, in that it conditions only on intrinsic properties of the tensor itself (the tensor rank and multilinear rank), as opposed to conditioning on a particular set of product tensors that decompose it. (Admittedly, Kruskal's theorem gives sufficient conditions for such a decomposition to constitute a unique tensor rank decomposition, which would then make these product tensors, in a sense, ``intrinsic" to the tensor they decompose.)

\begin{proof}[Proof of Corollary~\ref{ana_cor}]
Let $x_{n+a}=-y_a$ for each $a \in [n]$ for notational convenience. The assumption that the rank of $v \in \Lin(\V_j^*,V_{\hat{j}})$ equals $r_j$ for each $j \in [m]$ implies ${\dim\spn\{x_{a,j} : a \in [2n] \}\geq r_j}$ for each $j \in [m]$, so by Conjecture~\ref{conjecture} (or Theorem~\ref{s_cases} in the special cases), $\{x_1,\dots, x_{2n}\}$ splits, and hence is not minimal. Thus, there exists a subset $T \subseteq [2n]$ of size $1 \leq \abs{T} \leq 2n-1$ such that
\begin{align}
\sum_{a \in T} x_a =0.
\end{align}
It furthermore must hold that $\abs{T \cap [n]}=\abs{T \cap \{n+1,\dots,2n\}}$, for inequality would yield a decomposition of $v$ into a sum of less than $n$ product tensors, contradicting the fact that $v$ has tensor rank $n$. The result follows.
\end{proof}

%-----------------------------------------------------------------------------%
\section{Proving special cases of the main conjecture}\label{special_cases}

In this section we prove Theorem~\ref{s_cases}, which includes the bipartite, restricted tripartite, and restricted multipartite cases of Conjecture~\ref{conjecture}. In Proposition~\ref{bipartite_prop} and Theorem~\ref{tripartite_thm} we actually prove more general statements than the bipartite and restricted tripartite cases, respectively. By Propositions~\ref{field} and \ref{prop:min_split}, we can assume the underlying field is infinite and prove whichever version (non-minimal or splitting) of a special case of Conjecture~\ref{conjecture} is convenient. Proposition~\ref{bipartite_prop} is a straightforward consequence of Sylvester's rank inequality \cite{horn2013matrix}. The proofs of Theorems~\ref{tripartite_thm} and~\ref{sep_combo} are more involved, and use similar techniques to one another.

\begin{prop}[Bipartite case of Conjecture~\ref{conjecture}]\label{bipartite_prop}
Let $n \geq 2$ be an integer, let $\V=\V_1\otimes \V_2$ be a bipartite vector space over a field $\field$, and let
\begin{align}
\{x_a={x_{a,1}\otimes x_{a,2}}: a \in [n] \} \subset \pro{\V_1 : \V_2}
\end{align}
be a set of product tensors. For each $j \in [2]$, let
\begin{align}
d_j=\dim\spn\{{x_{a,j}} : a \in [n] \}.
\end{align}
If $n\leq d_1+d_2-1$, then
\begin{align}\label{bip}
\sum_{a \in [n]}  \alpha_a x_a \neq 0 \qquad \text{ for all } \qquad \alpha_1,\dots, \alpha_n \in \field^{\times}.
\end{align}
\end{prop}
Note that \eqref{bip} implies $\{x_1,\dots, x_n\}$ is not minimal, so Proposition~\ref{bipartite_prop} contains the bipartite case of Conjecture~\ref{conjecture}.
\begin{proof}[Proof of Proposition~\ref{bipartite_prop}]
It suffices to prove that $\sum_{a \in [n]} x_a \neq 0$, as the non-zero scalars $\alpha_1,\dots, \alpha_n$ can be absorbed into $x_1,\dots, x_n$. For each $j \in [2]$, let
\begin{align}
X_j=( x_{1,j}, \dots, x_{n,j}) \in \Lin(\field^n, \V_j).
\end{align}
Then,
\begin{align}
X_1 X_2^\t=\sum_{a \in [n]} x_{a,1} x_{a,2}^\t.
\end{align}
Thus,
\begin{align}
d_1+d_2& =\rank(X_1)+\rank(X_2^\t)\\
&\leq \rank(X_1 X_2^\t)+n,
\end{align}
where the second line is Sylvester's rank inequality \cite{horn2013matrix}.
Since $n \leq d_1+d_2-1$, this implies $\rank(X_1 X_2^\t)\geq 1$. By the isomorphism $\Lin(\V_2,\V_1) \cong \V_1 \otimes \V_2^*$ (where $\V_2^*$ is any dual space of $\V_2$), it holds that $\sum_a x_a \neq 0$,
which completes the proof.
\end{proof}
Now we prove a (more general statement than) the restricted tripartite case of Conjecture~\ref{conjecture}.

\begin{theorem}[Restricted tripartite case of Theorem~\ref{s_cases}]\label{tripartite_thm}
Let $n \geq 2$ be an integer, let ${\V=\V_1\otimes\V_2\otimes \V_3}$ be a tripartite vector space over a field $\field$, and let
\begin{align*}
\{x_a={x_{a,1}\otimes x_{a,2}}\otimes x_{a,3}: a \in [n] \} \subset \pro{\V_1 : \V_2:\V_3}
\end{align*}
 be a set of product tensors. For each $j \in [3]$, let
\begin{align}
d_j=\dim\spn\{{x_{a,j}} : a \in [n] \}.
\end{align}
If $d_3 \leq 2$ and $n \leq d_1+d_2$, then at least one of the following statements holds:
\begin{enumerate}
\item $\sum_{a \in [n]}  \alpha_a x_a \neq 0$ \qquad for all \qquad $\alpha_1,\dots, \alpha_n \in \field^{\times}$.
\item $\{x_{1,1},\dots,x_{n,1}\}$ splits, and $\{x_{1,2},\dots,x_{n,2}\}$ splits (possibly with respect to different partitions of $[n]$).
\end{enumerate}
\end{theorem}
\begin{remark} First note that Theorem~\ref{tripartite_thm} implies the restricted tripartite case of Conjecture~\ref{conjecture}. By Proposition~\ref{field} it suffices to consider the case that $\field$ is infinite. If Statement 1 holds, then $\{x_1,\dots, x_n\}$ splits. It remains to show that if $\{x_{1,2},\dots,x_{n,2}\}$ splits, then $\{x_1,\dots,x_n\}$ splits. This will follow from basic arguments that do not rely on Theorem~\ref{tripartite_thm}.

Let $S \sqcup T= [n]$ be a non-trivial partition such that
\begin{align}
\spn\{x_{1,2},\dots,x_{n,2}\}=\spn\{x_{a,2} : a \in S\} \oplus \spn\{x_{a,2} : a \in T\}.
\end{align}
We prove that $\{x_1,\dots, x_n\}$ also split with respect to $S\sqcup T$. Suppose toward contradiction there exists a non-zero tensor
\begin{align}
v \in \spn\{x_a : a \in S\} \cap \spn\{x_a : a \in T\}.
\end{align}
Then we can write
\begin{align}
v=\sum_{a \in S} \alpha_a x_a=\sum_{b \in T} \beta_b x_b
\end{align}
for some scalars $\alpha_a\in \field$ not all zero and scalars $\beta_b \in \field$ not all zero. Let $f_1 \in \V_1^*$ and $f_3 \in \V_3^*$ be non-zero linear functionals such that
\begin{align}
(f_1 \otimes \I \otimes f_3)v &\neq 0,
\end{align}
and for all $a \in [n]$, $f_1(x_{a,1}) \neq 0$ and $f_3({ x_{a,3}}) \neq 0$. Then
\begin{align}
(f_1 \otimes \I \otimes f_3)v=\sum_{a \in S} \alpha_a f_1({ x_{a,1}})f_3({ x_{a,3}}) x_{a,2}=\sum_{b \in T} \beta_b f_1({ x_{b,1}})f_3({ x_{b,3}}) x_{b,2}\in \V_2
\end{align}
is a non-zero vector in
\begin{align}
\spn\{x_{a,2} : a \in S\} \cap \spn\{x_{a,2} : a \in T\}.
\end{align}
This contradicts the fact that $\{x_{1,2},\dots,x_{n,2}\}$ splits with respect to $S \sqcup T$, and completes the proof.
\end{remark}
\begin{remark}
It is natural to ask whether the condition that $\{x_{1,3},\dots, x_{n,3}\}$ splits can be added to Statement 2 in Theorem~\ref{tripartite_thm}. It cannot, as evidenced by the example
\begin{align}
\{\;\pm \; e_1 \otimes e_1\otimes e_1,\;\pm \; e_2 \otimes e_2 \otimes e_2, \;\pm \; e_3 \otimes e_3 \otimes (e_1+e_2)\},
\end{align}
in which $6=n \leq d_1+d_2+d_3-2=6$, Statement 1 does not hold, and $\{x_{1,3},\dots, x_{n,3}\}$ does not split. It is also natural to ask whether Theorem~\ref{tripartite_thm} holds when $d_1, d_2,d_3 \geq 3$. It does not, as evidenced by the example
\begin{align*}
\{&\; \pm \; e_1 \otimes e_1 \otimes e_1,\; \pm \; e_2 \otimes e_2 \otimes e_2,\; \pm \; e_3 \otimes e_3 \otimes e_3,\\
&\; \pm \; e_4 \otimes (e_1+e_2+e_3) \otimes (e_1+e_2+e_3)\},
\end{align*}
in which $8=n\leq d_1+d_2+d_3-2=8$, Statement 1 does not hold, and only $\{x_{1,3},\dots, x_{n,3}\}$ splits. Also consider the example
\begin{align*}
\{&\; \pm \; e_1 \otimes e_1 \otimes e_1,\; \pm \; e_2 \otimes e_2 \otimes e_2,\; \pm \; e_3 \otimes e_3 \otimes e_3,\\
&\; \pm \; e_4 \otimes e_4 \otimes e_4,\; \pm \; u \otimes u \otimes u\},
\end{align*}
where $u=e_1+e_2+e_3+e_4$. In this case, $10=n \leq d_1+d_2+d_3-2=10$, Statement 1 does not hold, and $\{x_{1,j},\dots,x_{n,j}\}$ does not split for all $j \in [3]$.
\end{remark}
\begin{proof}[Proof of Theorem~\ref{tripartite_thm}]
It suffices to prove Theorem~\ref{tripartite_thm} with Statement 1 replaced by ${\sum_a x_a \neq 0}$, as the scalars $\alpha_1,\dots, \alpha_n \in \field^{\times}$ can be absorbed into $x_1,\dots, x_n$. We can assume $\field$ is infinite by Proposition~\ref{field}. We can also assume $d_3=2$, otherwise this reduces to the bipartite case. 

If $n<d_1+d_2$, let $f \in \V_3^*$ be any linear functional such that $f({ x_{a,3}})\neq 0$ for all $a \in [n]$. Then
\begin{align}
(\I \otimes \I \otimes f) \sum_{a \in [n]} x_a\neq 0
\end{align}
by Proposition~\ref{bipartite_prop}, so 
\begin{align}
\sum_{a \in [n]} {x_{a,1}\otimes x_{a,2}\otimes x_{a,3}}\neq 0.
\end{align}

It remains to consider the case $n=d_1+d_2$ and $d_3=2$. Suppose without loss of generality that $\{x_{1,1},\dots, x_{d_1,1}\}$ is linearly independent. If $\{x_{d_1+1,2},\dots, x_{d_1+d_2,2}\}$ is not linearly independent, then there exists $b \in [d_1]$ for which $x_{b,2} \notin \spn \{x_{d_1+1,2},\dots, x_{d_1+d_2,2}\}$. Let $f_2 \in \V_2^*$ be any linear functional such that $f_2({ x_{b,2}}) \neq 0$ and $f_2({x_{a,2}})=0$ for all ${a\in \{d_1+1,\dots, d_1+d_2\}}$. Then
\begin{align}
(\I \otimes f_2 \otimes \I) \sum_{a \in [n]} {x_{a,1}\otimes x_{a,2}\otimes x_{a,3}}&=\sum_{a\in [d_1]} f_2({ x_{a,2}}) x_{a,1} \otimes x_{a,3}\\
&\neq 0,
\end{align}
where the inequality follows from the fact that $\{x_{1,1},\dots, x_{d_1,1}\}$ is linearly independent. Thus,
\begin{align}
\sum_{a \in [n]}  {x_{a,1}\otimes x_{a,2}\otimes x_{a,3}}\neq 0.
\end{align}
Now suppose $\{x_{d_1+1,2},\dots, x_{d_1+d_2,2}\}$ is linearly independent, and assume toward contradiction that either $\{x_{1,1},\dots,x_{n,1}\}$ or $\{x_{1,2},\dots, x_{n,2}\}$ does not split, and
\begin{align}
\sum_{a \in [n]}  {x_{a,1}\otimes x_{a,2}\otimes x_{a,3}}= 0.
\end{align}
By symmetry, we can assume $\{x_{1,2},\dots, x_{n,2}\}$ does not split.

For each $a \in [d_1]$, it holds that
\begin{align}
x_{a,2}=\sum_{b=d_1+1}^{n} \alpha_{a,b} x_{b,2}
\end{align}
for some $\{\alpha_{a,b}\} \subseteq \field$. Let $S_a=\{b \in \{d_1+1,\dots,n\} : \alpha_{a,b}\neq 0\}$. We first  observe that for any $a \in [d_1]$, $b \in S_a$, it holds that ${x_{a,3} \in \spn \{x_{b,3}\}}$. Let $g_1 \in \V_1^*$ be any linear functional such that $g_1({ x_{a,1}}) \neq 0$ and $g_1({x_{c,1}})=0$ for all $c \in [d_1] \setminus \{a\}$, and let $g_2 \in \V_2^*$ be any linear functional such that $g_2({ x_{b,2}})\neq 0$ and $g_2({x_{c,2}}) =0$ for all $c \in [n]\setminus ([d_1] \sqcup \{b\})$. Then
\begin{align*}
(g_1 \otimes g_2 \otimes \I)\sum_{c \in [n]}  {x_{c,1}\otimes x_{c,2}\otimes x_{c,3}}&=g_1({x_{a,1}})g_2({ x_{a,2}})x_{a,3}+g_1({ x_{b,1}})g_2({ x_{b,2}})x_{b,3}\\
&=0.
\end{align*}
Note that
\begin{align}
g_2({ x_{a,2}})=\alpha_{a,b} g_2({ x_{b,2}})\neq 0,
\end{align}
so $g_1({ x_{a,1}})g_2({ x_{a,2}})\neq 0$. It follows that ${x_{a,3} \in \spn \{x_{b,3}\}}$, as claimed.

Consider the set
\begin{align}
S=\bigcup_{a \in [d_1]} (\{a\} \cup S_a).
\end{align}
Note that $S=[n]$, for otherwise
\begin{align}
\spn\{x_{a,2} : a \in [n]\}=\spn\{x_{a,2}:a \in S\} \oplus \{ x_{a,2} : a \in [n] \setminus S\},
\end{align}
contradicting the assumption that $\{x_{a,2} : a \in [n]\}$ does not split. For any $a,c \in [d_1]$, if $S_a \cap S_c \neq \{\}$, then $x_{b,3} \in \spn\{x_{a,3}\}$ for all $b \in S_a \cup S_c \cup \{a\} \cup\{c\}$. Since
\begin{align}
\dim\spn\{x_{a,3} : a \in [n]\} >1,
\end{align}
then there exists a non-trivial partition $T' \sqcup Q'=[d_1]$ such that the sets
\begin{align}
T&=\bigcup_{a \in T'} (\{a\} \cup S_a),\\
Q&=\bigcup_{a \in Q'} (\{a\} \cup S_a)
\end{align}
are disjoint (and partition $[n]$ by the fact that $S=[n]$). But this implies
\begin{align}
\spn\{x_{a,2} : a \in [n]\}=\spn\{x_{a,2}:a \in T\} \oplus \spn\{ x_{a,2} : a \in Q\}.
\end{align}
Indeed,
\begin{align}
\spn\{x_{a,2}:a \in T\}&=\spn\{x_{a,2} : a \in T \cap \{d_1+1,\dots, n\}\}\\
\spn\{x_{a,2}:a \in Q\}&=\spn\{x_{a,2} : a \in Q \cap \{d_1+1,\dots, n\}\},
\end{align}
and hence
\begin{align}
\spn\{x_{a,2}:a \in T\} \cap  \spn\{ x_{a,2} : a \in Q\}=\{0\}.
\end{align}
This contradicts the assumption that $\{x_{a,2} : a \in [n]\}$ does not split.
\end{proof}

Now we prove the restricted multipartite case of Conjecture~\ref{conjecture}.
\begin{theorem}[Restricted multipartite case of Conjecture~\ref{conjecture}]\label{sep_combo}
Let $n\geq 2$ and $m\geq 1$ be integers, let $\V=\V_1\otimes\dots\otimes\V_m$ be a multipartite vector space over a field $\field$, and let
\begin{align}
\{{x_a}: a \in [n] \} \subset \pro{\V_1 : \dots : \V_m}
\end{align}
be a set of product tensors. For each $j \in [m]$, let
\begin{align}
d_j=\dim\spn \{ x_{a,j}: a \in [n]\}.
\end{align}
If $n\leq \sum_{j=1}^m (d_j-1)+1$, $d_1 \geq 1$, and $1\leq d_2,\dots,d_m \leq 2$, then $\{x_a: a \in [n]\}$ splits.
\end{theorem}

To prove Theorem~\ref{sep_combo}, we require the following proposition.

%\begin{prop}\label{lemma:min_parti}
%Let $n\geq 2$ be an integer, let $\V$ be a vector space over a field $\field$, and let ${v_1,\dots, v_n \in \V}$ be non-zero vectors. If $v_1+\dots+v_n=0$, then there exists an integer $s \geq 1$ and a partition ${S_1 \sqcup \dots \sqcup S_s=[n]}$ such that for each $q \in [s]$,
%\begin{align}
%\sum_{a \in S_q} v_a=0,
%\end{align}
%and this sum constitutes a minimal linear dependence of $\{v_a : a \in S_q\}$.
%\end{prop}
%\begin{proof}
%We use induction on $n$. The base case $n=2$ is trivial. Proceeding inductively, if
%\begin{align}
%\sum_{a \in S} v_a \neq 0
%\end{align}
%for all $S \subseteq [n]$ of size $1 \leq \abs{S} <n$, then $v_1,\dots,v_n$ is minimal. Otherwise, there exists a nontrivial partition $S \sqcup T =[n]$ such that
%\begin{align}
%\sum_{a \in S} v_a=\sum_{a \in T}v_a=0.
%\end{align}
%By the induction hypothesis, $S$ and $T$ can be partitioned into minimal subsets, which then partitions $[n]$ into minimal subsets, proving the claim.
%\end{proof}
\begin{prop}\label{lemma:parti}
Let $n, s,$ and $t$ be positive integers. Let $\equiv$ be an equivalence relation on $[n]$, let $E$ be the set of equivalence classes of $[n]$, and let $S_1,\dots, S_s \subseteq [n]$ and $T_1,\dots, T_t \subseteq [n]$
be two collections of non-empty disjoint subsets that satisfy the following three conditions:
\begin{enumerate}
\item For all $q \in [s]$, every element of $S_q$ is equivalent modulo $\equiv$. Likewise, for all $r \in [t]$, every element of $T_r$ is equivalent.
\item Both collections partition $[n]$, i.e.
\begin{align}\label{partitionlemma}
\bigsqcup_{q\in [s]} S_q =\bigsqcup_{r \in [t]}T_r=[n].
\end{align}
\item For any two subsets $Q \subseteq [s], R \subseteq [t]$, if
\begin{align}
\bigsqcup_{q \in Q} S_q = \bigsqcup_{r \in R} T_r=N
\end{align}
for some subset $N \subseteq [n]$, then $N \in \{\{\}, [n]\}$.
\end{enumerate}
Then $\equiv$ is trivial, i.e. $E=\{[n]\}$.
\end{prop}
\begin{proof}
Let $N\in E$ be an equivalence class. By conditions 1 and 2, there exist $Q \subseteq [s]$ and $R \subseteq [t]$ such that
\begin{align}
\bigsqcup_{q \in Q} S_q = \bigsqcup_{r \in R} T_r=N.
\end{align}
Condition 3 implies $N=[n]$, completing the proof.
\end{proof}

\begin{proof}[Proof of Theorem~\ref{sep_combo}]
It suffices to prove that $\{x_1,\dots, x_n\}$ is not minimal in the case when $\field$ is infinite. We use induction on $n$. The base case $n=2$ is trivial. Proceeding inductively, suppose toward contradiction that $\{x_1,\dots,x_n\}\subseteq \pro{\V_1: \cdots: \V_m}$ satisfy ${n\leq \sum_{j=1}^m (d_j-1)+1}$ and are minimal. By absorbing the coefficients of the minimal linear dependence into each product tensor, we may assume
\begin{align}\label{eq:zeroo}
\sum_{a \in [n]} x_a=0,
\end{align}
and this constitutes a minimal linear dependence of $\{x_1,\dots,x_n\}$.

Define an equivalence relation $\equiv$ on $[n]$ by $a \equiv b$ if $x_{a,1}\in \spn\{x_{b,1}\}$. Let $E$ be the set of equivalence classes. For each $A \in E$, let $\Pi_A \in \Lin(\V_1)$ be any operator with ${\ker(\Pi_A)=\spn\{x_{a,1}\}}$ for all $a \in A$. Applying $(\Pi_A \otimes \I)$ to~\eqref{eq:zero} gives
\begin{align}
\sum_{a \in [n] \setminus A} (\Pi_A \otimes \I)x_a=0.
\end{align}
Note that every product tensor $(\Pi_A \otimes \I)x_a$ appearing in this sum is non-zero. Let ${S_1^A,\dots,S_{s_A}^A \subseteq [n]\setminus A}$ be a set of non-empty disjoint subsets that partition $[n] \setminus A$, i.e.
\begin{align}
S_1^A \sqcup \dots \sqcup S_{s_A}^A = [n]\setminus A,
\end{align}
and such that for all $q \in [s_A]$, it holds that
\begin{align}
\sum_{a \in S_q^A} (\Pi_A \otimes \I)x_a=0,
\end{align}
and constitutes a minimal linear dependence of $\{x_a : a \in S_q^A\}$.
%The existence of such a partition follows from Proposition~\ref{lemma:min_parti}. 
For each $A \in E$, $q \in [s_A]$, define
\begin{align}
d_{q,j}^A=
\begin{cases}
\dim\spn\{\Pi_A x_{a,1} : a \in S_q^A\}, & j=1.\\
\dim\spn\{x_{a,j} : a \in S_q^A\},& j>1.
\end{cases}
\end{align}
By the induction hypothesis,
\begin{align}\label{eq:projector_version}
\abs{S^A_q} \geq \sum_{j=1}^m (d^A_{q,j}-1)+2.
\end{align}
Subtracting this inequality from $n\leq \sum_{j=1}^m (d_j-1)+1$ gives
\begin{align}\label{eq:sumdifd}
\sum_{j=1}^m (d_j-d_{q,j}^A) \geq n-\abs{S_q^A}+1.
\end{align}
\begin{claim}\label{claim:parallel}
There exists an index $j \in \{2,\dots,m\}$ and equivalence classes $A \neq B \in E$ such that $d_{q,j}^A < d_j$ and $d_{r,j}^B<d_j$ for all $q \in [s_A]$, $r \in [s_B]$. In particular, $d_j=2$ and $d_{q,j}^A=d_{r,j}^B=1$ for all $q \in [s_A]$, $r \in [s_B]$.  
\end{claim}
\begin{proof}[Proof of claim]
\renewcommand\qedsymbol{$\triangle$}

For each $A \in E$, $q \in [s_A]$, define
\begin{align}
J_q^A= 
\bigsqcup_{j=2}^m 
\begin{cases}
\{j_1,\dots,j_{d_j-d_{q,j}^A}\}, & d_j-d_{q,j}^A \geq 1,\\
\{\}, & d_j-d_{q,j}^A \leq 0.
\end{cases}
\end{align}
(Essentially, $J_q^A$ is a multiset containing each $j \in \{2,\dots, m\}$ with multiplicity ${\max\{0, d_j-d_{q,j}^A\}}$, but regarded as a set by adding subscripts.) To prove the claim, it suffices to find $A \neq B \in E$ such that
\begin{align}
\Big(\bigcap_{q\in [s_A]} J_q^A\Big) \cap \Big(\bigcap_{r\in [s_B]} J_r^B\Big)\neq \{\}.
\end{align}
First note that
\begin{align}
\biggabs{\bigcap_{q \in [s_A]} J^A_q} & \geq \sum_{q \in [s_A]} \abs{J^A_q}-(s_A-1)\biggabs{\bigcup_{r \in [s_A]} J^A_r}\\
&\geq  \abs{A} + \sum_{q \in [s_A]} d_{q,1}^A -(d_1-t_A)\\
&\geq \abs{A}. \label{thirdline}
\end{align}
The first line follows from the standard result that for any two sets $J_1,J_2$,
\begin{align}
\abs{J_1 \cap J_2} = \abs{J_1}+\abs{J_2}-\abs{J_1+J_2},
\end{align}
and an inductive argument. The second line follows from the inequality
\begin{align}\label{eq:junionbound}
\biggabs{\bigcup_{\substack{A \in E\\r \in [s_A]}} J^A_r} \leq n-d_1,
\end{align}
along with
\begin{align}\label{eq:jbound}
\abs{J_q^A} & \geq \sum_{j=2}^m (d_j-d_{q,j}^A)\\
& \geq n-\abs{S_q^A}+d_{q,1}^A-d_1+1,\label{eq:jbound2}
\end{align}
and algebraic simplification. The inequality~\eqref{eq:junionbound} follows from the fact that for all $A \in E$, $q \in [s_A]$,
\begin{align}\label{jcontainment}
J_q^A \subseteq \bigsqcup_{j=2}^m \{j_1,\dots, j_{d_j-1}\},
\end{align}
and $n \leq \sum_{j=1}^m (d_j-1)+1$. The containment~\eqref{jcontainment} follows from the fact that $d_{q,j}^A \geq 1$ for all $j \in [m]$. The inequality~\eqref{eq:jbound} follows from the definition of $J_q^A$. The inequality~\eqref{eq:jbound2} follows from~\eqref{eq:sumdifd}. The third line~\eqref{thirdline} follows from
\begin{align}
\sum_{q \in [s_A]} d_{q,j}^A &=\sum_{q \in [s_A]} \dim\spn\{\Pi_A x_{a,1}: a \in S_q^A \}\\
&\geq \dim\spn\{\Pi_A x_{a,1}: a \in [n]\setminus A \}\\
&= \dim\spn\{\Pi_A x_{a,1}: a \in [n] \}\\
&= \dim\spn\{x_{a,1}: a \in [n] \setminus A \}-1\\
&\geq d_1-1.
\end{align}
Here, the first line is by definition. The second line follows from the standard result that for subspaces $W_1,\dots,W_n \subseteq U$ of a vector space $U$, it holds that
\begin{align}
\sum_{a \in [n]} \dim{W_a} \geq \dim\bigg[\sum_{a\in [n]} W_a\bigg].
\end{align}
The third line follows from $\Pi_A x_{a,1}=0$ for all $a \in A$. The fourth line follows from $\dim\ker(\Pi_A)=1$, and the fifth line is by definition.

The inequality~\eqref{thirdline} implies
\begin{align}
\sum_{A \in E} \biggabs{\bigcap_{q \in [s_A]} J^A_q} \geq n.
\end{align}
The claim follows from~\eqref{eq:junionbound} and the pigeonhole principle.
\end{proof}
Fix an index $j \in \{2,\dots,m\}$ and equivalence classes $A \neq B \in E$ as in Claim~\ref{claim:parallel} for the remainder of the proof, so that $d_j=2$ and $d^A_{q,j}=d^B_{r,j}=1$ for all $q \in [s_A]$, $r \in [s_B]$. Define an equivalence relation $\sim$ on $[n]$ by $a \sim b$ if $x_{a,j}\in \spn\{x_{b,j}\}$. To complete the proof, we use Proposition~\ref{lemma:parti} to conclude that $\sim$ is trivial, and hence
\begin{align}
\dim\spn\{x_{a,j}:a \in [n]\}=1,
\end{align}
which contradicts $d_j=2$ and completes the proof. Note that the partitions
\begin{align}
\bigsqcup_{a \in A} \{a\} \sqcup \bigsqcup_{q \in [s_A]} S^A_q=\bigsqcup_{b \in B} \{b\} \sqcup \bigsqcup_{r \in [s_B]} S^B_r=[n].
\end{align}
satisfy the conditions 1 and 2 of Lemma~\ref{lemma:parti}. For condition 3, suppose there exist subsets $\mathscr{A} \subseteq A$, $\mathscr{B} \subseteq B$, $Q \subseteq [s_A]$, and $R \subseteq [s_B]$ such that
\begin{align}
\mathscr{A} \sqcup \bigsqcup_{q \in Q} S^A_q=\mathscr{B}\sqcup \bigsqcup_{r \in R} S^A_r = N
\end{align}
for some subset $N \subseteq [n]$. Then
\begin{align}
\sum_{a \in N} x_a \in \ker(\Pi_A \otimes \I) \cap \ker(\Pi_B \otimes \I)=\{0\},
\end{align}
so $N \in \{\{\}, [n]\}$ by the fact that $\sum_{a \in [n]} x_a$ constitutes a minimal linear dependence of $\{x_1,\dots, x_n\}$. This completes the proof.
\end{proof}

%-----------------------------------------------------------------------------%
\section{The inequality appearing in the main conjecture is sharp}\label{conjecture_kruskal}

In this section,
% we prove that the inequality $n \leq \sum_{j=1}^m (d_j-1)+1$ appearing in Conjecture~\ref{conjecture} cannot be weakened,
we find a set of product tensors that does not split and satisfies ${n = \sum_{j=1}^m (d_j-1)+2}$. In fact, we prove that this set of product tensors forms a circuit, which is stronger than not splitting. This proves that the bound in Corollary~\ref{linincor}, and the inequality $n \leq \sum_{j=1}^m (d_j-1)+1$ appearing in Conjecture~\ref{conjecture}, are both sharp. The example we use is Derksen's~\cite{DERKSEN2013708}, which he uses to prove that the inequality appearing in Kruskal's theorem is sharp in a similar sense.
\begin{prop}
For any field $\field$ with greater than $n$ elements, and positive integers $d_1,\dots, d_m$ with $n=\sum_{j=1}^m (d_j-1)+2$, there exist vector spaces $\V_1,\dots,\V_m$ over $\field$ and a set of product tensors $\{x_1,\dots,x_n\}\subseteq \pro{\V_1 : \cdots : \V_m}$ that forms a circuit, and satisfies
\begin{align}
\dim\spn\{x_{a,j} : a \in [n]\}\geq d_j
\end{align}
for all $a \in [n]$.
\end{prop}
\begin{proof}
By Theorem 2 of \cite{DERKSEN2013708}, there exist vector spaces $\V_1,\dots, \V_m$ over $\field$, a positive integer $p \leq n$, and product tensors $\{x_a : a \in [p]\}\subseteq \pro{\V_1 : \cdots : \V_m}$ with k-ranks $d_j=\krank(x_{1,j},\dots, x_{p,j})$ such that ${\sum_{a\in [p]} x_a=0}$. If $p < n$, then $p \leq \sum_{j=1}^m (d_j-1)+1$, which implies $\{x_a:a\in[p]\}$ is linearly independent by Theorem~\ref{connect}. But this contradicts $\sum_{a \in [p]} x_a=0$, so $p=n$. The equality $n=\sum_{j=1}^m (d_j-1)+2$ implies that $d_j \leq n-1$ for all $j \in [m]$. It follows that for any subset $S \subseteq [n]$ of size $\abs{S}=n-1$, it holds that $\krank(x_{a,j} : a \in S) \geq d_j$. Since $n-1= \sum_{j=1}^m (d_j-1)+1$, then by Theorem~\ref{connect},  $\{x_a:a\in S\}$ is linearly independent. It follows that $\{x_a : a \in [n]\}$ is a circuit.
\end{proof}

%-----------------------------------------------------------------------------%
\bibliographystyle{alpha}
\bibliography{tensor_rank}
%-----------------------------------------------------------------------------%

\end{document}